%% file: ex_article.tex
\documentclass[final,onefignum,onetabnum]{siamonline171218}

\newif\ifarxiv
\arxivtrue
\newif\ifpublication
\ifarxiv
  \publicationfalse
\else
  \publicationtrue
\fi
\newcommand{\versioncite}[2]{\ifarxiv\cite{#1}\else\cite{#2}\fi}

\input{ex_shared}
\ifpdf
\hypersetup{
  pdftitle={Exact affine conditioning beyond Gaussians: a unique characterization of the ensemble Kalman update},
  pdfauthor={Frederic J.~N.~Jorgensen, Youssef Marzouk}
}
\fi

\ifpublication
\myexternaldocument{ex_supplement}
\fi

\begin{document}

\maketitle

\begin{abstract}
The analysis step of the \fjedit{stochastic} ensemble Kalman filter, called the ensemble Kalman update (EnKU), is widely used for approximating posterior distributions in inverse problems and data assimilation. The EnKU approximates the posterior distribution $\pi_{X\mid Y=y_\star}$ by pushing forward the joint distribution $(X,Y)\sim\pi$ through an affine map $L^{\mathrm{EnKU}}_{\pi,y_\star}(x,y)$ that depends only on the covariance structure of $\pi$ and the observation $y_\star$. While the EnKU yields the exact posterior for Gaussian $\pi$ in the mean-field, this property alone does not uniquely determine the EnKU. In fact, there  are infinitely many affine maps $L_{\pi, y_\star}$ that achieve such exact conditioning. 
In this paper, we offer a novel characterization of the EnKU among all such affine maps. 
We first exhaustively  characterize the set \(\mathcal{E}^{\mathrm{EnKU}}\) of joint distributions for which the EnKU yields exact conditioning, showing that it is much larger than the set of Gaussians. Next, we show that except for a small class of highly symmetric distributions within \(\mathcal{E}^{\mathrm{EnKU}}\), the EnKU is the \emph{unique} exact affine conditioning map. Further, we characterize the largest possible set of distributions \(\mathcal{F}\) for which a distribution-dependent, weakly observation-dependent, affine map exists, a class of transports that naturally includes the EnKU. We show that \(\mathcal{F}=\mathcal{E}^{\mathrm{EnKU}}\cup\mathcal{S}_{\mathrm{nl-dec}}\) with a small symmetry class \(\mathcal{S}_{\mathrm{nl-dec}}\), meaning that for affine conditioning beyond the Gaussian setting, the EnKU has an exact set that is essentially maximally large.
\end{abstract}

\begin{keywords}
 Ensemble Kalman filter; stochastic filtering; measure transport; Bayesian inverse problems; uncertainty quantification; mean-field limit; non-Gaussian setting; exact conditioning; data assimilation. 
\end{keywords}

\begin{AMS}
  65C35, 62F15, 93E11
\end{AMS}

\section{Introduction}
Conditioning a joint probability distribution on the realized value of one variable is a fundamental operation in inverse problems and stochastic filtering. Given a joint distribution $(X,Y)\sim \pi$ on $\mathbb{R}^n \times \mathbb{R}^m$ and an observation $y_\star \in \mathbb{R}^m$, the conditional distribution $\pi_{X\mid Y=y_\star}$ is defined via disintegration. 
At the population level, many conditioning algorithms for filtering and inverse problems construct a transport map $T_{\pi, y_\star}:\R^n \times \R^m \rightarrow \R^n$ whose pushforward of the joint distribution approximately matches the conditional distribution \cite{evensen2003ensemble, spantini2022coupling, calvello2025ensemble, bishop2001adaptive, tippett2003ensemble, el2012bayesian, kaveh2026amortizedenergybasedbayesianinference, al2023nonlinear}, i.e.,
\begin{equation}
\label{eq:transport_measures}
(T_{\pi, y_\star})_\sharp \pi \approx \pi_{X\mid Y=y_\star}.
\end{equation}
\fjedit{The stochastic EnKF and its nonlinear extensions are important instances of this construction. After constructing joint samples \((X_i,Y_i)\), they transport these joint samples into approximate posterior samples at the realized value \(y_\star\).} In high dimensions, using an affine map $L_{\pi, y_\star}$ for $T_{\pi,y_\star}$ is particularly attractive \fjedit{for this task} because it can be implemented using standard linear algebra operations such as matrix-vector products and low-rank updates that scale efficiently with dimension. The transport map most widely used in practice, particularly in data assimilation and inverse problems, is indeed affine; it is the update step of the \fjedit{stochastic} ensemble Kalman filter (EnKF), denoted by $T_{\pi, y_\star} = L^{\mathrm{EnKU}}_{\pi, y_\star}$, which is well-defined under the assumption that $\pi$ has finite second moments. This map takes the form
\begin{equation}
\label{eq:enkf_def}
  L^{\mathrm{EnKU}}_{\pi, y_\star}(x,y)
  = x + K(y_\star - y),
  \qquad 
  K = \Sigma_{XY}\Sigma_{YY}^\dagger,
\end{equation}
where $\Sigma_{XY}$ denotes the cross-covariance between $X$ and $Y$, $\Sigma_{YY}$ is the covariance of $Y$, and $^\dagger$ denotes the Moore--Penrose pseudoinverse \cite{evensen2003ensemble, evensen2009ensemble}. In the  data assimilation literature, the matrix $K$ is often   referred to as the \emph{Kalman gain}. We will call the transport map $L^{\mathrm{EnKU}}_{\pi, y_\star}$  the \textit{ensemble Kalman update} (EnKU). 
\subsection{Ambiguity of the ensemble Kalman update}
A special property of the EnKU is that when $X$ and $Y$ are jointly Gaussian, this update is \emph{exact}. That is, the approximation in \eqref{eq:transport_measures} becomes an equality; i.e., 
\[
\left(L^{\mathrm{EnKU}}_{\pi, y_\star}\right)_\sharp \pi = \pi_{X\mid Y=y_\star},
\]
holds $\pi_Y$-a.s.\ for $y_\star \in \R^m$ \cite{evensen2003ensemble, evensen2009ensemble, reich2015probabilistic, calvello2025ensemble}.  However, this property does not single out the EnKU among affine maps.  Indeed, as we will explain more later, there are infinitely many affine maps $L_{\pi, y_\star}: \R^{n}\times\R^m \rightarrow \R^n$ achieving $(L_{\pi, y_\star})_\sharp\pi =\pi_{X\mid Y=y_\star}$ for Gaussian $\pi$. Why, then, the particular choice of $K=\Sigma_{XY}\Sigma_{YY}^{\dagger}$, and in what sense is the EnKU  preferable? As it turns out, answering this question requires an analysis in the \textit{non-Gaussian} regime. In this paper, we consider the conditioning of non-Gaussian distributions to provide a new characterization of the ensemble Kalman update, uniquely identifying it among affine maps through the exactness of its pushforward distributions.

Before proceeding, we emphasize an important distinction from the existing literature. Prior work shows that the Kalman update is variance-minimizing among linear unbiased \textit{point estimators} $c + BY$ of $X$; i.e., the choice $B = K$, $c = \E(X) - K \E(Y)$ is optimal in this class. It is the so-called best linear unbiased estimator (BLUE) \cite{gelb1974applied, calvello2025ensemble}.  This notion of optimality is fundamentally different from our characterization, which is based on \emph{distributional exactness} and, in particular, considers the EnKU, which outputs an ensemble instead of a point estimate.
\subsection{Summary of main results}
In order to formalize our uniqueness claim, note that the EnKU takes a pair \(({\pi}, y_\star)\) and returns an affine map \(L^{\mathrm{EnKU}}_{{\pi}, y_\star}\). As such, it belongs to a broader class of transports that we term \emph{affine conditioning maps}. Below and throughout the paper, $\Pp_2(\mathbb{R}^d)$ denotes the space of Borel probability distributions on $\mathbb{R}^d$ with finite second moment. 
\begin{definition}
\label{def:affine_cond_map}
An \emph{affine conditioning map} is a mapping
\[
L:\ \mathcal{P}_2(\mathbb{R}^n\times\mathbb{R}^m)\times\mathbb{R}^m 
\longrightarrow \{\text{affine maps }\mathbb{R}^n\times\mathbb{R}^m\rightarrow\mathbb{R}^n\},\qquad
(\pi,y_\star)\longmapsto L_{\pi,y_\star},
\]
such that each $L_{\pi,y_\star}$ admits the affine representation
\[
L_{\pi,y_\star}(x,y)
= A(\pi,y_\star)x + B(\pi,y_\star)y + c(\pi,y_\star),
\]
where $A(\pi,y_\star)\in\mathbb{R}^{n\times n}$, $B(\pi,y_\star)\in\mathbb{R}^{n\times m}$, and $c(\pi,y_\star)\in\mathbb{R}^{n}$.
\end{definition}
It is clear that the EnKU map $L^{\mathrm{EnKU}}: (\pi,y_\star)\longmapsto L^{\mathrm{EnKU}}_{\pi, y_\star}$  defined in Equation \eqref{eq:enkf_def}  is an affine conditioning map. 
Note, however, that $A$, $B$, and $c$ as in Definition \ref{def:affine_cond_map} are much more general and allowed to depend on all of $\pi$, particularly on all of its moments.  We will often write $L_{ \pi, y}$ for $L$ to make this dependence explicit. 

Next, to distinguish the EnKU among affine maps, we define what it means to condition exactly. 
\begin{definition} 
\label{def:exact_conditioning}
Let $\pi \in \Pp_2(\R^n \times \R^m)$, $y_\star \in \R^m$, and fix a version of the Markov kernel $y \mapsto \pi_{X\mid Y=y}$. We say that an affine map  $\ell(x, y) := Ax + By + c$ with fixed $A\in \R^{n\times n}, B \in \R^{n \times m}, c \in \R^n$ \emph{conditions exactly} at $y_\star$ for $\pi$ if
$$
\ell_\sharp \pi = \pi_{X|Y=y_\star}. 
$$
We say that an affine conditioning map $L$  {conditions exactly at $y_\star$ for $\pi$}  if $\ell := L_{\pi,y_\star}$ conditions exactly at $y_\star$ for $\pi$.  
Further, if  $\pi_Y$-a.s.\ in $y_\star$ it holds that $L$ conditions exactly at $y_\star$ for $\pi$,    then we say that $L$ is an {exact conditioning map for $\pi$}. This is abbreviated by ``$L$ is exact for $\pi$'' or simply ``$L$ is exact'' if $\pi$ is clear from the context. 
\end{definition}
\fjedit{Note that when \(\ell=L_{\pi,y_\star}\), i.e., \(\ell\) arises from an affine conditioning map \(L\), the fixed coefficients \(A,B,c\) in Definition~\ref{def:exact_conditioning} are the values \(A(\pi,y_\star),B(\pi,y_\star),c(\pi,y_\star)\) from Definition~\ref{def:affine_cond_map}.} 
Further, note that exact affine conditioning at $y_\star$ for $\pi$ requires a choice of the Markov kernel $\pi_{X|Y=y}$, and we will only invoke this definition if such a choice was made beforehand.  Exactness of $L_{\pi,y}$  for $\pi$, on the other hand, is independent of the choice of Markov kernel $\pi_{X|Y=y}$.

A simple covariance calculation shows that $L^\mathrm{EnKU}_{\pi,y_\star}$ is indeed an exact conditioning map for any Gaussian $\pi$. However, there are infinitely many other affine  conditioning maps $L_{\pi, y_\star}(x,y)=A(\pi,y_\star)x + B(\pi, y_\star)y + c(\pi, y_\star)$ that are also exact for Gaussians. For example, for a fixed Gaussian $\pi$ and $y_\star \in \R^m$ and for every choice of $F \in \R^{n \times m}$ (assuming $\Cov(X+FY)$ has full rank), there are choices $D\in\R^{n\times n}$  and $c\in \R^n$ such that $\ell(x,y)= D(x + Fy) + c$ is exact. This is a simple consequence of the fact that $X + FY$ is Gaussian and there is an affine transport map between any two non-singular Gaussians. \fjedit{For a complete characterization of exact affine conditioning maps for \textit{Gaussian} distributions, see \cite[Appendix~C.1]{calvello2025ensemble}.}  
The natural question, then, is why the EnKU should be preferred among all such affine maps that achieve exact conditioning. As we will show, this choice can be justified by studying exact conditioning in the \emph{non-Gaussian} regime.  To understand what distinguishes the EnKU among affine conditioning maps (and what does not), we study the \emph{exact set} of the EnKU, which we define for a general affine conditioning map $L$ as 
\begin{equation}
\label{eq:exact_set}
\mathcal{E}(L) := \left\{\pi \in \Pp_2\left(\R^n\times\R^m\right) \, \left | \, \,   L \text{ is an exact affine  conditioning map for }\pi \right. \right\}.
\end{equation}
We answer the following two questions in this paper, with formal results in Section \ref{section:enkf_set_analysis}:
\begin{enumerate}
    \item  What is the set of distributions $\pi \in \Pp_2(\R^n\times\R^m)$ for which the EnKU update $L^{\mathrm{EnKU}}$ is an exact conditioning map, i.e., what is the exact set $\mathcal{E}^\mathrm{EnKU}:= \mathcal{E}(L^\mathrm{EnKU})$?
    \item Given a distribution $\pi\in \mathcal{E}^\mathrm{EnKU}$  for which the EnKU is exact and an observation $y_\star\in\R^m$, is the EnKU update $L^{\mathrm{EnKU}}_{\pi, y_\star}$ the \textit{only} affine map  achieving exact affine conditioning?
\end{enumerate}
The first question is answered in Proposition \ref{prop:exact_set_kalman}. Importantly, $\mathcal{E}^{\mathrm{EnKU}}$ is substantially larger than the class of Gaussian distributions \fjedit{and contains all joint laws of the form \(X=Z+KY\) with \(Z \indep Y\), including non-Gaussian \(Z\) and \(Y\)}.
In Theorem \ref{thm:enkf_update_unique}, we answer the second question: excluding certain strongly symmetric distributions \fjedit{(including Gaussians)}, given $\pi \in \mathcal{E}^\mathrm{EnKU}$ and $y_\star \in \R^m$, the EnKU update $L_{\pi, y_\star}^{\mathrm{EnKU}}$ is the \emph{only} affine map  that is exact for $\pi$ at $y_\star$. 

Conversely, when choosing an affine conditioning map $L$ from the infinitely many possibilities, to reduce bias we may want an $L$ for which the set $\mathcal{E}(L)$ is maximally large.  In Section \ref{section:beyond_enkf}, to study this maximally large set, we  define \emph{weakly observation-dependent affine conditioning maps}  as affine conditioning maps $L$ of the form 
$$
L_{\pi,y_\star}(x,y)=A(\pi)x + B(\pi)y + c(\pi, y_\star),
$$
generalizing commonly used affine conditioning maps like the EnKU or square root updates \cite{evensen2003ensemble,evensen2009ensemble, tippett2003ensemble, nerger2012unification}. 
\begin{table}[htbp]
\centering
\renewcommand{\arraystretch}{1.3}%
\setlength{\tabcolsep}{5pt}%
\begin{tabular}{c l c c l}
\toprule
 & \fjedit{Map} & \fjedit{$A(\pi,y_\star)$} & \fjedit{$B(\pi,y_\star)$} & \fjedit{$c(\pi,y_\star)$} \\
\midrule
\rowcolor{cmweakrow}
  & \fjedit{$L^{\mathrm{EnKU}}$ \eqref{eq:enkf_def}} & \fjedit{$I$} & \fjedit{$-K$} & \fjedit{$K\,y_\star$} \\
\rowcolor{cmweakrow}
  & \fjedit{$L^{\mathrm{D}}$ \eqref{eq:map_D}} & \fjedit{$S_{\mathrm D}$} & \fjedit{$0$} & \fjedit{$m_X - S_{\mathrm D}\,m_X + K(y_\star-m_Y)$} \\
\rowcolor{cmweakrow}
\multirow{-3}{*}{\shortstack{\textcolor{cmweakedge}{\textbf{\fjedit{weakly obs.-dep.}}}\\[1pt]
  \textcolor{cmweakedge}{\fjedit{(only $c$ depends}}\\[-1pt]
  \textcolor{cmweakedge}{\fjedit{on $y_\star$ \eqref{eq:weakly_y_star})}}}}
  & \fjedit{$L^{\mathrm{OT}}$ \eqref{eq:map_OT}} & \fjedit{$S_{\mathrm{OT}}$} & \fjedit{$0$} & \fjedit{$m_X - S_{\mathrm{OT}}\,m_X + K(y_\star-m_Y)$} \\
\midrule
\rowcolor{cmstrongrow}
\textcolor{cmstrongedge}{\textbf{\fjedit{fully obs.-dep.}}}
  & \shortstack[l]{\fjedit{Learned filter}\\[1pt]\fjedit{(\Cref{remark:learned_ensemble_filters})}} & \fjedit{$I$} & \fjedit{$B_\theta(\pi,y_\star)$} & \fjedit{$c_\theta(\pi,y_\star)$} \\
\bottomrule
\end{tabular}
\caption{\fjedit{Affine coefficients $A,B,c$ of representative affine conditioning maps $L_{\pi,y_\star}(x,y)=A(\pi,y_\star)\,x+B(\pi,y_\star)\,y+c(\pi,y_\star)$ appearing throughout the paper. Here $K=\Sigma_{XY}\Sigma_{YY}^\dagger$ as in \eqref{eq:enkf_def}, $(m_X,m_Y)$ is the mean of $\pi$, $S_{\mathrm D}=\Sigma_{X|Y}^{1/2}\Sigma_X^{\dagger/2}$, and $S_{\mathrm{OT}}=\Sigma_X^{\dagger/2}(\Sigma_X^{1/2}\Sigma_{X|Y}\Sigma_X^{1/2})^{1/2}\Sigma_X^{\dagger/2}$.}}
\label{tab:affine_maps}
\end{table}
We investigate the size of the largest possible exact set  that any  weakly observation-dependent $L$ might have, which turns out to take the form
$$
\mathcal{F}:=\bigcup\limits_{L \text{ weakly observation-dep.}}\mathcal{E}(L).
$$
In Theorem \ref{thm:F_char} we  show that the EnKU is  exact on all of $\mathcal{F}$ except for pathological counterexamples, thereby almost achieving the maximal level of exactness any weakly observation-dependent affine conditioning map can have. More formally, we show that there is a small symmetry class $\mathcal{S}_{\mathrm{nl-dec}}\subseteq \mathcal{F}$ such that 
$$
\mathcal{F} = \mathcal{E}^\mathrm{EnKU}\cup\mathcal{S}_{\mathrm{nl-dec}}.
$$
This negative result concerning the potential enlargement of the EnKU exactness class through weakly observation-dependent affine transports is consistent with recent developments in learned ensemble filters \cite{mccabe2021learning, revach2022kalmannet, bach2025learning}, where the Kalman gain (corresponding to the matrix $B$) was explicitly made observation-dependent to enable conditioning on more general distributions. In particular, those learned maps were not merely weakly observation-dependent, allowing them to go beyond $\mathcal{F}$ (see also Remark \ref{remark:learned_ensemble_filters}).

\subsection{A comment on ensemble approximations}
In practical implementations, quantities depending on the population distribution $\pi$ are typically approximated through the empirical distribution
\[
\widehat{\pi}^N \:= \frac{1}{N}\sum_{i=1}^N \delta_{(X_i,Y_i)},
\qquad (X_i,Y_i)\sim \pi,
\]
associated with an ensemble of $N$ particles. 
This leads to an empirical approximation of the conditional distribution,
\[
\widehat{\pi}^N_{X\mid Y=y_\star} \:= \left(L^{\mathrm{EnKU}}_{\widehat{\pi}^N, y_\star}\right)_\sharp \widehat{\pi}^N.
\]
In this paper, however, we carry out our analysis in the \emph{mean-field} setting, operating at the level of population distributions $\pi$ and the corresponding mean-field conditional distribution
\[
\tilde{\pi}_{X \mid Y = y_\star} := \left(L^{\mathrm{EnKU}}_{\pi, y_\star}\right)_\sharp \pi.
\]
This perspective is standard in the theoretical analysis of ensemble-based methods and provides a principled baseline for understanding their behavior \cite{reich2015probabilistic, calvello2025ensemble}. To justify this reduction mathematically, let $d$ denote a probability metric on $\Pp_2(\R^n)$. Then the discrepancy between the empirical posterior and the true conditional distribution admits the decomposition
\begin{equation}
\label{eq:variance_bias_decomposition}
d\left(\widehat{\pi}^N_{X\mid Y=y_\star}, \pi_{X\mid Y=y_\star}\right)
\le
\underbrace{d\left(\widehat{\pi}^N_{X\mid Y=y_\star}, \tilde \pi_{X\mid Y=y_\star}\right)}_{\text{variance}}
+
\underbrace{d\left(\tilde \pi_{X\mid Y=y_\star}, \pi_{X\mid Y=y_\star}\right)}_{\text{bias}}.
\end{equation}
The first ``variance'' term captures sampling error arising from a finite ensemble, while the second term reflects intrinsic bias of the population-level transport. To control the first term, classical concentration \cite{kallenberg1997foundations, hall2014martingale, feller1971introduction, le2009large} and techniques from propagation of chaos \cite{chaintron2022propagation, sznitman2006topics} can be used to show that, given sufficient regularity, the variance is $\mathcal O(N^{-\alpha})$
for some parameter $\alpha> 0$ and a particular choice of metric $d$; see, e.g., \cite{jorgensen2026quantitativewassersteinpropagationchaos,le2009large, al2024non} for examples of such bounds. 
This decomposition shows that in regimes where the ensemble size $N$ is large, the \emph{bias} term dominates the error, making it natural to ask under what conditions this bias vanishes exactly---which is precisely the question addressed in this paper. We will also demonstrate this effect in our numerical experiments (Section \ref{section:num_exp}). 
Since our analysis is carried out in the mean-field setting, important practical issues related to finite sample size---such as localization, covariance inflation, and direct analysis of small-ensemble behavior---are outside the scope of this work. Another important limitation of our analysis is that it is \emph{single-step} in nature, focusing on one step of (likelihood-free) Bayesian inference rather than multi-step filtering dynamics; as a result, we do not address issues such as long-term stability or error accumulation over time. Both categories of issues have been extensively studied elsewhere \cite{reich2015probabilistic, anderson2007adaptive, houtekamer2001sequential, law2015data, tong2015nonlinear, del2018stability, van2009uniform}.
\subsection{Notation}
\begin{table}[H]
\centering
\begin{tabular}{p{0.17\textwidth}p{0.73\textwidth}}
\toprule
\fjedit{Symbol} & \fjedit{Meaning} \\
\midrule
\fjedit{\(\mathcal E(L)\)} & \fjedit{Exact set of a general affine conditioning map \(L\) (equation~\ref{eq:exact_set})} \\
\fjedit{\(\mathcal E^{\mathrm{EnKU}}\)} & \fjedit{Exact set of the EnKU (\Cref{prop:exact_set_kalman})} \\
\fjedit{\(\mathcal S_{\mathrm{cov}},\mathcal S_{\mathrm{dec}},\mathcal S_{\mathrm{cyc}}\)} & \fjedit{Symmetry classes inside \(\mathcal E^{\mathrm{EnKU}}\) where other exact affine maps may exist (\Cref{def:symmetry_classes})} \\
\fjedit{\(\mathcal F\)} & \fjedit{Maximal exact set over weakly observation-dependent affine conditioning maps (equation~\ref{eq:F_def})} \\
\fjedit{\(\mathcal S_{\mathrm{nl-dec}}\)} & \fjedit{Exceptional class satisfying \(\mathcal F=\mathcal E^{\mathrm{EnKU}}\cup\mathcal S_{\mathrm{nl-dec}}\) (\Cref{def:nl_dec})} \\
\bottomrule
\end{tabular}
\caption{\fjedit{Distribution classes used throughout the paper.}}
\label{tab:exactness_classes}
\end{table}
 For $d\in\mathbb{N}$, we always consider $\R^d$ with inner product $\langle\cdot,\cdot\rangle$ and Euclidean norm $\|\cdot\|_2$. $I_d$, or simply $I$, is the $d\times d$ identity matrix. For a matrix $A$, $A^\top$ is the transpose, $A^\dagger$ the Moore--Penrose pseudoinverse, and $\sqrt{A}$ and $A^{\frac{1}{2}}$ both denote the {principal} symmetric square root when $A\succeq 0$. For an endomorphism/square  matrix $A$, we refer to the spectrum as $\sigma(A)$. $\rho(A)$ is the spectral radius of such a matrix. $\mathrm{GL}(n)$ is the general linear group.  The Frobenius norm is written as $\|\cdot\|_{\mathrm F}$. \(R_\theta\in\mathbb{R}^{2\times 2}\) denotes the two-dimensional rotation matrix \(R_\theta=\begin{pmatrix}\cos\theta&-\sin\theta\\ \sin\theta&\cos\theta\end{pmatrix}\), which rotates vectors in \(\mathbb{R}^2\) counterclockwise by angle \(\theta\).  
For a random vector $X$, its law/distribution is $\Law(X)\in \Pp_2(\R^d)$, expectation $\E(X)$, covariance $\Cov(X)=\E \left ((X-\E X)(X-\E X)^\top \right )$, and centered version $\overline X:=X-\E X$.  For a joint distribution $\pi\in \Pp_2(\R^n\times\R^m)$,  $\pi_X \in \Pp_2(\R^n)$ and $\pi_Y \in \Pp_2(\R^m)$ denote the marginals on the first $n$ and last $m$ coordinates.  $\pi_{X\mid Y=y}$ is a (fixed) version of the conditional distribution (a Markov kernel).  Letting $(X, Y) \sim \pi$, define the cross-covariance $\Sigma_{XY}(\pi) := \mathrm{Cov}(X,Y)$ and the auto-covariance $\Sigma_{YY}(\pi) := \mathrm{Cov}(Y)$, both under $\pi$. We omit the dependence on $\pi$ when it is clear from context. We say that $X_1\stackrel{d}{=}X_2$ for random vectors $X_1, X_2$ if they have the same distribution.  The Dirac probability measure at $x$ is $\delta_x$. $T_\sharp \mu$ denotes the pushforward of a distribution $\mu$ by  a measurable map $T$. $W_2$ is the $2$-Wasserstein distance on $\Pp_2(\R^d)$. $\mathcal S^{\mathrm{c}}$ is the complement in $\mathcal A$ of a subset $\mathcal S\subseteq \mathcal A$.

\section{Characterizing the \fjedit{stochastic} EnKF update }
\label{section:enkf_set_analysis}
As described in the introduction, it is well known that the EnKU, defined in \eqref{eq:enkf_def} \fjedit{as the analysis step of the stochastic EnKF}, is exact for Gaussian distributions \cite{reich2015probabilistic, kalman1960new}. In this section, we will go beyond Gaussian distributions by (i) identifying the set of distributions $\mathcal{E}^\mathrm{EnKU} =\mathcal{E}(  L^{\mathrm{EnKU}})$ (recall \eqref{eq:exact_set}) on which the EnKU is exact; and (ii) understanding the structure of transport maps that are exact for some element of $\mathcal{E}^\mathrm{EnKU}$. \fjedit{Table~\ref{tab:exactness_classes} summarizes the distribution classes used in this and the next section.}

The EnKU works by taking every $x$-sample and correcting it linearly with its corresponding increment $K(y_\star -y)$. This reveals the underlying structure of the EnKU: rather than operating on a Gaussian assumption per se, it assumes that there is a linear relationship between $X$ and $Y$. Informally, if we can approximately expand 
$$
X \approx Z + KY + \mathcal{O}(Y^2)
$$
for $Z$ independent of $Y$, \fjedit{where the remainder $\mathcal{O}(Y^2)$ collects the nonlinear part of the dependence of $X$ on $Y$ and is small in the sense that the joint law of $(X,Y)$ is close to that of $(Z+KY,Y)$},
then the EnKU will yield accurate results. \fjedit{If this nonlinear remainder vanishes, so that \(X=Z+KY\) with \(Z\) independent of \(Y\), then the EnKU is exact. The following proposition, whose proof is given in the appendix, formalizes this idea, completely characterizing all distributions in $\mathcal{E}^\mathrm{EnKU}$.}
\begin{proposition}
\label{prop:exact_set_kalman}
Let $\mathrm{Lin}$ be  the class of linear maps from $\R^n\times \R^m$ to $\R^n$. Then the following equation fully characterizes the exact set of the EnKU:
\begin{align}
\mathcal{E}^\mathrm{EnKU}= \Bigl\{ \pi \in \Pp_2\left(\R^n\times\R^m\right)  \, |& \,  \,  \exists \pi_{X|Y}, \nu \in \Pp_2\left(\R^n\right), O \in \mathrm{Lin}  \label{eq:exact_set_kalman}   \\
& \text{ s.t. } \pi_{X|Y=y} = O(\cdot,y)_\sharp \nu  \  \forall y\in\R^m \Bigl\}.\notag
\end{align}
\end{proposition}
\fjedit{Thus \(\mathcal E^{\mathrm{EnKU}}\) contains many non-Gaussian exemplars: if \(Z\sim\nu\) and \(Y\sim\pi_Y\) are independent and \(X=Z+KY\), then \((X,Y)\in\mathcal E^{\mathrm{EnKU}}\) for \emph{any} \(\nu\) and \(\pi_Y\) that have finite second moments (which need not be Gaussian). Section~\ref{section:num_exp}, especially Figure~\ref{fig:densities}, illustrates conditionals of such distributions with multimodal and ring-shaped laws of \(Z\).
We make several remarks on this result.}
\begin{remark}
\fjedit{The representation in Proposition~\ref{prop:exact_set_kalman} in terms of \(\nu\) and \(O\) is not in general unique. For example, for any \(U\in\mathrm{GL}(n)\), replacing \(\nu\) by \(U_\sharp\nu\) and \(O(z,y)\) by \(O(U^{-1}z,y)\) leaves all pushforwards \(O(\cdot,y)_\sharp\nu\) unchanged. We will adopt the following canonical choice in this manuscript. Let \(K=\Sigma_{XY}\Sigma_{YY}^{\dagger}\) be the Kalman gain from \eqref{eq:enkf_def}, define \(Z=X-KY\), set \(\nu=\Law(Z)\), and let \(O(z,y)=z+Ky\). The ``\(\subseteq\)'' part in the proof of Proposition~\ref{prop:exact_set_kalman} shows that, for \(\pi\in\mathcal E^{\mathrm{EnKU}}\), this construction uniquely fixes a valid version of the Markov kernel; whenever we write \(\pi_{X\mid Y=y_\star}\) below, we refer to the choice with distribution given by \(Z+Ky_\star\).}
\end{remark}
\begin{remark}
Note that the inclusion $\supseteq$ in the set equality above  is closely related to \cite[Remark 4]{spantini2022coupling}, where exact and generally nonlinear transport-based conditioning is characterized from an independence viewpoint. Specifically, the authors remark that if a map $S\in\mathrm{Lin}$ is invertible in the $x$-component, and satisfies $S(X,Y)\indep  Y$ for $(X,Y)\sim\pi$, then the map $\ell := S(\cdot,y_\star)^{-1}\circ S$ implements exact conditioning, i.e., $\ell_\sharp\pi=\pi_{X\mid Y=y}$. The independent-rendering choice $S(x,y)=x-\Sigma_{XY}\Sigma_{YY}^\dagger y$ recovers precisely the EnKU update.  Motivated by this perspective, in settings where nonlinear structure is present---specifically, when there exist nonlinear features $\phi$ such that $X \stackrel{d}{=} \phi(Z,Y)$ for $Z \indep Y$---natural extensions of the EnKU have been proposed and studied, including the conditional mean filter (corresponding to $\phi(Z,Y)=Z+f(Y)$) \cite{lei2011moment, hoang2021machine} and the stochastic map filter (where $\phi$ has triangular structure \cite{spantini2022coupling}). 
\end{remark}
\begin{remark}  
\fjedit{In inverse problems and filtering, the joint law of $(X,Y)$ is often specified through an additive independent-noise model $Y=h(X)+\varepsilon$ with $X\indep\varepsilon$ and a measurable map $h: \R^n \rightarrow \R^m$. \ifarxiv Appendix\else Supplement\fi~\ref{section:additive_noise_models} restates the EnKU exactness condition of Proposition~\ref{prop:exact_set_kalman} directly in terms of a linear second-order partial differential equation involving the forward map $h$ and the log of the densities $p_X$ and $p_\varepsilon$ of the prior and noise, assuming their existence. For generic $h$ and noise $\varepsilon$, the joint law need not lie in $\mathcal{E}^{\mathrm{EnKU}}$. Nonetheless, a large exact set is desirable as it keeps the EnKU bias small for laws \emph{near} $\mathcal{E}^{\mathrm{EnKU}}$. \ifarxiv Appendix\else Supplement\fi~\ref{section:bias_outside_exact_set} reports a numerical inverse-problem example in which the EnKU bias remains smaller than that of competing affine updates as the joint law is moved progressively outside $\mathcal{E}^{\mathrm{EnKU}}$.}
\end{remark}
The question we answer in the remainder of this section is whether there can be \textit{other} affine transports
$$
L_{\pi,y_\star} (x,y) = A(\pi,y_\star)x +  B(\pi, y_\star)y  + c(\pi, y_\star),
$$
besides the EnKU, that achieve exactness for $\pi \in \mathcal{E}^\mathrm{EnKU}$. 
To gain some intuition, we go back to the set of Gaussian $\pi$, which is clearly contained in $\mathcal{E}^\mathrm{EnKU}$. As we explained in the introduction, there are many other affine maps implementing exact conditioning for Gaussian joint distributions $\pi$. The fundamental reason for this degree of freedom in the choice of $L$ is that Gaussian distributions have strong symmetries. The distribution of a Gaussian vector  is a \emph{stable distribution}, meaning that the sum of two independent Gaussians is again Gaussian \cite{samorodnitsky1994stable, zolotarev1986one, nolan1998multivariate}.\footnote{Gaussians are the only stable random variables with finite second moment \cite{feller1971introduction}.}  As a consequence, Gaussians are \emph{self-decomposable}: specifically, letting $G$ denote a Gaussian random vector in $\R^d$, $G$ is \emph{$\lambda$-decomposable} for every $\lambda\in (0,1)$ \cite{loeve1945nouvelles, ken1999levy, rajba2001multiple}, meaning that there exists another independent Gaussian vector $G_\lambda$ such that
$$
G \stackrel{d}{=}\lambda G + G_\lambda.
$$
Another strong symmetry possessed by non-degenerate \fjedit{centered }Gaussian vectors $G$ is a rescale-then-rotate symmetry:  there is a matrix $C \in \R^{d\times d}$ (e.g., the inverse of any square root of the covariance matrix) such that the distribution of $C\overline G$ is invariant under any rotation. \fjedit{This same rotational non-uniqueness appears in ensemble adjustment-type square-root filters, where the posterior covariance condition fixes the analysis square root only up to an orthogonal factor (see \hyperref[section:connection_esrf]{\ifarxiv Appendix\else Supplement\fi~\ref*{section:connection_esrf}} for the precise finite-ensemble formulation).} In the following theory, we will demonstrate that it is due to these symmetries that there are many possible choices of exact affine conditioning maps for Gaussians. A third symmetry leading to many possible choices of conditioning maps is when $Z\sim\nu$ corresponding to some $\pi \in \mathcal{E}^\mathrm{EnKU}$  has constant components, meaning that $v^\top Z$ is a.s.\ constant for some $v \neq 0$ (equivalently, that $Z$ has a singular covariance matrix).

Generalizing these three symmetries (namely, $\lambda$-de\-com\-pos\-abi\-li\-ty of the joint distribution,  the rescale-then-rotate symmetry, and singular covariance matrices) to non-Gaussian joint distributions leads to the final EnKU characterization result presented in Theorem \ref{thm:enkf_update_unique}. 
\begin{definition}
\label{def:symmetry_classes}
\fjedit{We define the sets \(\mathcal{S}_{\mathrm{dec}}, \mathcal{S}_{\mathrm{cyc}}, \mathcal{S}_{\mathrm{cov}} \subseteq \mathcal{E}^\mathrm{EnKU}\). Let \(\pi \in \mathcal{E}^\mathrm{EnKU}\), and write \((X,Y)=(Z+KY,Y)\sim\pi\) with \(K=\Sigma_{XY}\Sigma_{YY}^{\dagger}\), \(Y\sim\pi_Y\), \(Z\sim\nu\), and \(Z\indep Y\).}
\begin{enumerate}[label=(\roman*)]
\item \(\pi \in \mathcal{S}_{\mathrm{dec}}\) if and only if there exist complex vectors \(v\in\C^n \setminus \{0\}\), \(w\in\C^m\), and a constant \(\lambda \in \mathbb{C}\), \(|\lambda| < 1\), such that
\[
v^\top \overline Z \stackrel{d}{=} \sum\limits_{k = 0}^\infty \lambda^k w^\top  \overline Y^{(k)}
\]
for i.i.d.\ copies \(\overline Y^{(k)}\) of \(\overline Y\).
\item \(\pi \in \mathcal{S}_{\mathrm{cyc}}\) if and only if there exist real vectors \(v_1, v_2 \in \R^n \setminus \{0\}\) such that, with \(Z_{\mathrm{cyc}} =(v_1^\top Z, v_2^\top Z)^\top\),
\[
\overline Z_{\mathrm{cyc}} \stackrel{d}{=} R_{2\pi/k} \overline  Z_{\mathrm{cyc}}
\]
for some \(k \in \N_{\geq 2}\), where \(R_\theta\) is the two-dimensional rotation by angle \(\theta\).
\item \(\pi \in \mathcal{S}_{\mathrm{cov}}\) if and only if \(\nu\) has singular covariance.
\end{enumerate}
\end{definition}
We make two remarks about this definition.
\begin{remark}
    Note that $\mathcal{S}_{\mathrm{cov}}\subset \mathcal{S}_{\mathrm{dec}}$ by choosing $v \in \mathrm{Ker}(\Cov(Z)) \setminus \{0\}$ and $w= 0$.
\end{remark}
\begin{remark}
\label{rem:gaussians_connection}
    We clarify how Gaussian distributions relate to the three symmetry classes $\mathcal{S}_{\mathrm{cov}}$, $\mathcal{S}_{\mathrm{dec}}$, and $\mathcal{S}_{\mathrm{cyc}}$. Let $(X,Y)\sim \pi=\mathcal N(\mu,\Sigma)$ be jointly Gaussian and write $(X,Y)=(Z+KY,Y)$ with $K=\Sigma_{XY}\Sigma_{YY}^\dagger$. Then the covariance of $Z$ is the Schur complement $\Cov(Z)=\Sigma_{XX}-\Sigma_{XY}\Sigma_{YY}^{\dagger}\Sigma_{YX},$ which coincides with the posterior covariance. In particular, $\pi\in\mathcal{S}_{\mathrm{cov}}$ if and only if the posterior distribution is singular.  Next, consider the class $\mathcal{S}_{\mathrm{dec}}$. If $\Cov(Z)$ is singular, then $\pi\in\mathcal{S}_{\mathrm{dec}}$ (see the previous remark). If $\Cov(Z)$ is non-singular, then $\pi\in\mathcal{S}_{\mathrm{dec}}$ if and only if $Y$ is non-constant, by additive closedness of Gaussians. Finally, $\pi\in\mathcal{S}_{\mathrm{cyc}}$ always holds for Gaussian distributions since after centering, Gaussians are invariant under sign flips, corresponding to cyclic symmetry of order $k=2$.
\end{remark}
Within $\mathcal{E}^{\mathrm{EnKU}}$, each of the symmetry classes above carves out a highly non-generic and {``small''} subset of distributions. {Smallness here can be quantified by the number   of constraints defining the set.} If $\pi\in \mathcal{S}_{\mathrm{cov}}$, then $\nu$ has singular covariance, so $Z$ lives almost surely in a proper affine subspace of $\R^n$. Let $\pi\in \mathcal{S}_{\mathrm{dec}}$. Then one linear functional of $\overline Z$ is a geometrically weighted infinite linear combination of a \emph{single} functional of $\overline Y$ {in distribution, imposing infinitely many scalar constraints.} In particular, this identity forces $\lambda$-decomposability of $v^\top\overline{Z}$, which is a special non-generic property \cite{loeve1945nouvelles, ken1999levy, rajba2001multiple}. A simple way of seeing that $\lambda$-decomposability for a random variable $U$ with characteristic function $\phi_U$ is easily violated  is to note that the defining equation $\phi_U(t) = \phi_U(\lambda t)\phi_{U_\lambda}(t)$ is unsatisfiable for  many characteristic functions with zeros (e.g., uniform distribution, atoms, etc.). Further, $\pi \in \mathcal{S}_{\mathrm{dec}}$ forces the decomposition variable to be a projection $w^\top \overline{Y}$, imposing a strong self-similar convolution equation on the joint distribution. If $\pi\in \mathcal{S}_{\mathrm{cyc}}$, there are $v_1,v_2\neq 0$ so that the two-dimensional projection $Z_{\mathrm{cyc}}=(v_1^\top Z,v_2^\top Z)$ is invariant under the finite rotation group $\{R_{2\pi m/k}\}_{m=0}^{k-1}$, imposing strong symmetry constraints. 
\begin{remark}
\label{rem:small_symmetry}
Alternatively, smallness of the sets $\mathcal{S}_{\mathrm{cov}}, \mathcal{S}_{\mathrm{dec}}, \mathcal{S}_{\mathrm{cyc}}$ can be understood topologically when ignoring the case of singular observation covariance.  Define
\[
\mathcal E^{\mathrm{EnKU}}_0
:=
\bigl\{
\pi \in \mathcal E^{\mathrm{EnKU}}\big|\mathrm{Cov}(\pi_Y)\text{ is invertible}\bigr\},
\]
and endow it with the relative \(W_2\)-topology. Then, $\mathcal S_{\mathrm{dec}} \cap \mathcal E^{\mathrm{EnKU}}_0$ and $\mathcal S_{\mathrm{cyc}} \cap \mathcal E^{\mathrm{EnKU}}_0$ are meagre subsets of \(\mathcal E^{\mathrm{EnKU}}_0\). In particular, $\mathcal S_{\mathrm{cov}}$ is a meagre subset of \(\mathcal E^{\mathrm{EnKU}}_0\). \ifarxiv A formal version of this statement appears below in Appendix~\ref{sec:small_symmetry_topological}.\else A formal version of this statement appears in the supplement \ref{sec:small_symmetry_topological}.\fi
\end{remark}
Although the EnKU is exact on each of these symmetry classes, the proof of the following theorem reveals that there are many other affine conditioning maps that are also exact on $\mathcal{S}_{\mathrm{cov}}$, $\mathcal{S}_{\mathrm{dec}}$, or $\mathcal{S}_{\mathrm{cyc}}$. Yet the theorem also shows that as soon as our distribution violates one of these symmetries, the space of possible affine filters contracts sharply. 
\begin{figure}[htbp]
    \centering
    \includegraphics[width=\columnwidth]{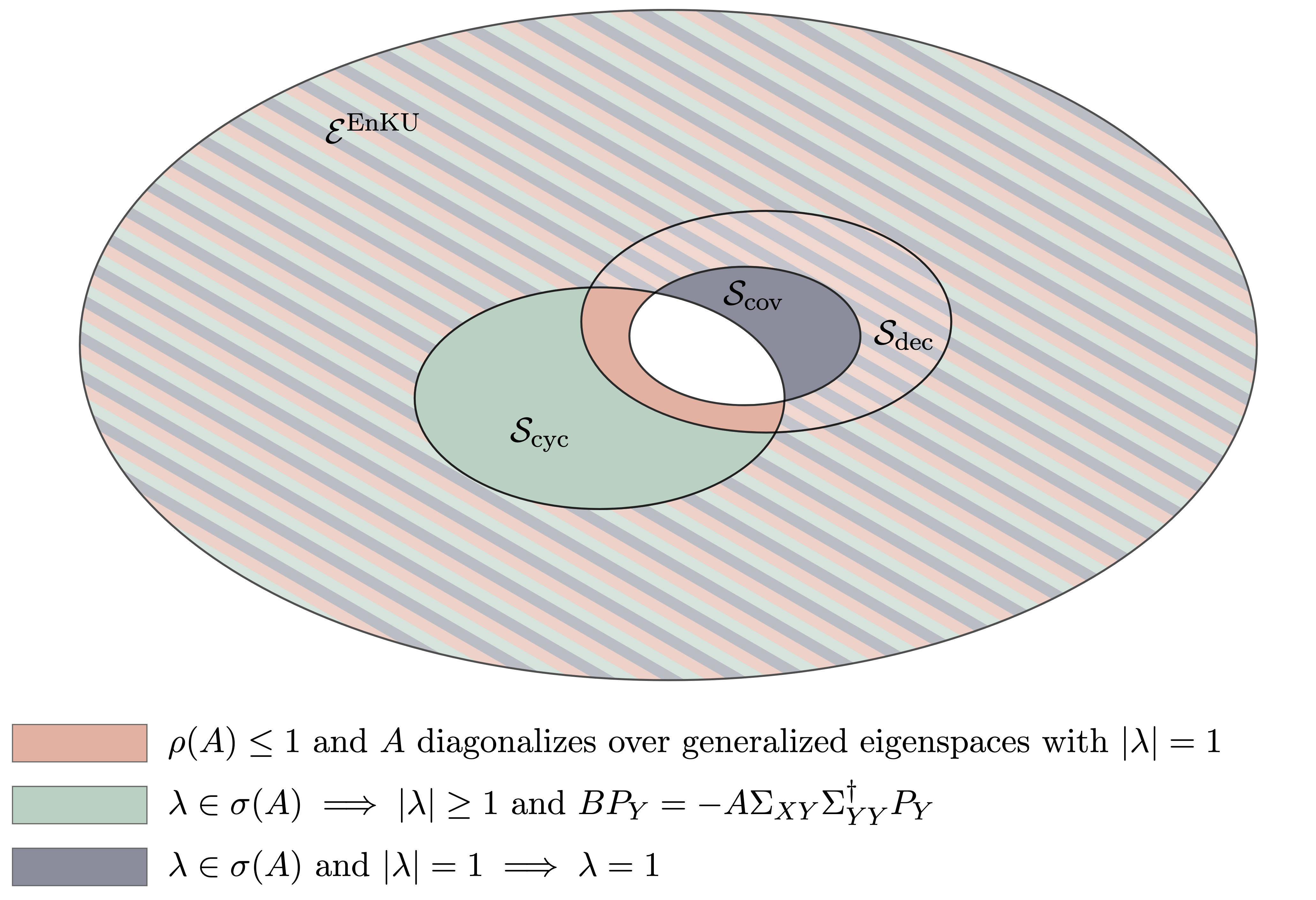}
    \caption{Theorem \ref{thm:enkf_update_unique} shows for any given $\pi\in \mathcal{E}^\mathrm{EnKU}$ that for any symmetry  $\mathcal{S}_{\mathrm{cov}}, \mathcal{S}_{\mathrm{dec}}, \mathcal{S}_{\mathrm{cyc}}$ it violates, strong structural constraints are imposed on any affine conditioning map $Ax + By +c$.  By Corollary \ref{cor:enkf_char}, if it violates all these symmetries, $Ax + By +c$ must be the EnKU. This corresponds to the region outside $\mathcal{S}_{\mathrm{cov}}$, $\mathcal{S}_{\mathrm{dec}}$, and $\mathcal{S}_{\mathrm{cyc}}$ in the diagram.}
    \label{fig:thm_enkf_char}
\end{figure}
\begin{theorem}
\label{thm:enkf_update_unique}
Consider $\pi \in \mathcal{E}^\mathrm{EnKU}$. Pick some $y_\star \in \R^m$ and assume that $\ell(x,y) = Ax + By +c$ conditions exactly for $\pi$ at $y_\star$. 
\begin{enumerate}
    \item  If $\pi \notin \mathcal{S}_{\mathrm{cov}}$, then $\rho(A) \leq 1$, and {$A$ is diagonalizable on the direct sum of generalized eigenspaces corresponding to eigenvalues $\lambda \in \C$ with $|\lambda|=1$.} 
    \item    If $\pi \not \in \mathcal{S}_{\mathrm{dec}}$, then   $A$ has no complex eigenvalues with magnitude smaller than $1$ and $$BP_Y = -A\Sigma_{XY}\Sigma_{YY}^\dagger P_Y,$$ where $P_Y  = \Sigma_{YY}\Sigma_{YY}^\dagger$ is the orthogonal projector onto the column space of $\Sigma_{YY} :=\mathrm{Cov}(Y)$.  
\item If $\pi \not \in \mathcal{S}_{\mathrm{cyc}}$, then $A$ has no complex eigenvalues with $|\lambda| = 1$ and  $\lambda \neq 1$. 
\end{enumerate}
\end{theorem}
\fjedit{The proof, given in the appendix, rewrites the distributional exactness condition for a candidate affine map \(Ax+By+c\), using the representation \(X=Z+KY\), as constraints on \(A\), \(B\), and \(c\). The three symmetry classes above are precisely the ways these constraints can fail to determine the EnKU coefficients uniquely.} Note that the third point, $\pi \not \in \mathcal{S}_{\mathrm{cyc}}$, can \emph{never} occur for Gaussian distributions, while the second, $\pi \not \in \mathcal{S}_{\mathrm{dec}}$, arises only in pathological edge cases (see Remark~\ref{rem:gaussians_connection}), underscoring that our characterization fundamentally relies on non-Gaussian structure. The following corollary, also proved in the appendix, explains that if a distribution violates all three of these symmetries, the only possible exact affine update is the EnKU. To rule out spurious  offsets in the constant \(c\), we assume \(\Sigma_{YY}\) is invertible. This is natural: zero-variance directions of \(Y\) carry no sample-dependent information about \(X\) and can be projected out a priori.
\begin{corollary}
 \label{cor:enkf_char}Consider $\pi \in \mathcal{E}^\mathrm{EnKU}$ with non-singular covariance $\Sigma_{YY}$.  Pick some $y_\star \in \R^m$ and assume that $\ell(x,y) = Ax + By +c$ is an exact affine transport for $\pi$ at $y_\star$. If $\pi \not \in \mathcal{S}_{\mathrm{cov}}$, $\pi \not \in \mathcal{S}_{\mathrm{dec}}$, and $\pi \not \in \mathcal{S}_{\mathrm{cyc}}$, then $\ell$ is the EnKU: 
$$
\ell(x,y) =  L_{\pi,y_\star}^{\mathrm{EnKU}}(x,y).
$$
\end{corollary}
This is a unique characterization result of the EnKU. As the set of symmetry-free distributions is the largest part of $\mathcal{E}^\mathrm{EnKU}$, \Cref{cor:enkf_char} supports defaulting to the EnKU to avoid bias within $\mathcal{E}^\mathrm{EnKU}$.

\section{Beyond the ensemble Kalman update}
\label{section:beyond_enkf}
The previous section established that, apart from a small symmetry class $\mathcal{S}_{\mathrm{cov}} \cup \mathcal{S}_{\mathrm{dec}}  \cup \mathcal{S}_{\mathrm{cyc}} $, the ensemble Kalman update (EnKU) is the unique affine conditioning map that is exact for any element  $\pi \in\mathcal{E}^\mathrm{EnKU}$ and observation $y_\star\in\R^m$. This observation raises a natural concern: perhaps the restriction to   $\mathcal{E}^\mathrm{EnKU}$ is too limiting. If one were to consider  different affine conditioning maps $L_{\pi,y_\star}$, could the associated exact set $\mathcal{E}(L_{\pi,y_\star})$ be strictly larger than $\mathcal{E}^\mathrm{EnKU}$? In other words, is it possible to design an update rule that is exact for a much broader class of distributions, thereby outperforming the EnKU in terms of bias reduction?  
\subsection{Maximal exactness of weakly observation-dependent affine conditioning maps}
To investigate this possibility, we extend our analysis to the family of \emph{weakly observation-dependent affine conditioning maps}  
taking the form
\begin{equation}
\label{eq:weakly_y_star}
L_{\pi,y_\star}(x,y) = A(\pi)x + B(\pi)y + c(\pi,y_\star).
\end{equation}
Our restriction to this class is motivated by practice: these maps are general enough to cover most update rules practically used in high-dimensional ensemble-based data assimilation  \cite{evensen2003ensemble,evensen2009ensemble, tippett2003ensemble, nerger2012unification}. In particular, they encompass commonly used deterministic alternatives such as square root updates. Therefore, weakly observation-dependent affine conditioning maps provide a natural framework in which to ask whether moving beyond the EnKU can substantially enlarge the domain of exactness. Define
\begin{equation}
\label{eq:F_def}
\mathcal{F} := \bigcup_{L \text{ weakly observation-dep.}}\mathcal{E}(L).
\end{equation}
The set $\mathcal{F}$ is the maximal exact set achievable by any single weakly observation-dependent affine update.  \fjedit{The proof of the following proposition, given in the appendix, uses that \(A(\pi)\) and \(B(\pi)\) do not vary with \(y_\star\), and changing \(y_\star\) can therefore only translate the pushforward law through the offset \(c(\pi,y_\star)\).}
\begin{proposition}
\label{prop:char_weak_set}
Let $\pi \in\mathcal{F}$. Then there exists a Markov kernel $\pi_{X|Y=y}$, a measurable $d:\R^m\rightarrow \R^n$, and $ \nu\in\Pp_2\left(\R^n\right)$ such that 
$$
\pi_{X|Y=y_\star} = T_{d(y_\star)}\nu\text{ for all }y_\star \in \R^m
$$
where we define the translation operator on distributions $T_{h}: \Pp_2\left(\R^n\right) \rightarrow \Pp_2\left(\R^n\right)$ through $T_h \mu := (x \mapsto x+h)_\sharp \mu$ for every $h \in \R^n$. Also, $d$ is $\pi_Y$-a.s.\ unique up to jointly translating $\nu$ and $d$. 
\end{proposition}
Before stating the main result of this section, we introduce the class $\mathcal{S}_{\mathrm{nl-dec}} \subseteq \mathcal{F}$, where nl-dec denotes ``nonlinearly decomposable.'' 
\begin{definition}
\label{def:nl_dec}
We define $\mathcal{S}_{\mathrm{nl-dec}}\subseteq \mathcal{F}$. Let $\pi \in \mathcal{F}$  and let $(\nu, d)$ be defined  as in Proposition \ref{prop:char_weak_set}. Set $Z \sim \nu$ and $Y \sim \pi_Y$ independently. Then we say  $\pi \in \mathcal{S}_{\mathrm{nl-dec}}$ if and only if  there exist complex vectors $v \in\C^n\setminus\{0\}$, $w \in \C^m,u\in\C^n$, and constants $\lambda \in \mathbb{C}$, $|\lambda| < 1$, $b \in \C$ such that 
\begin{equation}
\label{eq:S_nl_dec_rec}
v^\top \overline Z \stackrel{d}{=} \sum\limits_{k = 0}^\infty \lambda^k \bigl (w^\top   Y^{(k)} + u^\top d(Y^{(k)}) \bigr ) +  b
\end{equation}
for i.i.d.\ copies $\{Y^{(k)}\}_{k\geq 0}$ of $Y$. 
\end{definition}
\fjedit{This non-generic class tying a one-dimensional nonlinear feature of the $Y$-marginal to a linear functional of $Z$ resembles $\mathcal{S}_{\mathrm{dec}}$ from the previous section, and we conjecture that it is topologically ``small'' in the same sense (see \ifarxiv Appendix\else Supplement\fi~\ref{sec:small_symmetry_topological}). }
\begin{theorem}
\label{thm:F_char}
The set of all $\pi \in \Pp_2\left(\R^n\times \R^m\right)$ that have a weakly observation-dependent exact affine update decomposes as 
$$
\mathcal{F} = \mathcal{E}^\mathrm{EnKU} \cup \mathcal{S}_{\mathrm{nl-dec}}.
$$
\end{theorem}
\fjedit{The proof, given in the appendix, shows that outside \(\mathcal S_{\mathrm{nl-dec}}\), the exactness equations force \(d\) in Proposition~\ref{prop:char_weak_set} to be affine in \(Y\); hence \(\pi\in\mathcal E^{\mathrm{EnKU}}\) by Proposition~\ref{prop:exact_set_kalman}.} Thus, weak observation dependence leaves essentially no room to beat the EnKU: the maximal exact set collapses to $\mathcal{E}^{\mathrm{EnKU}}$ up to the small symmetry class $\mathcal{S}_{\mathrm{nl-dec}}$.  Practically, unless this special nonlinear decomposability can be exploited, any weakly observation-dependent affine rule cannot exceed the exactness domain of the EnKU; in this sense, the EnKU is \textit{almost} optimal among weakly observation-dependent affine conditioning maps.
Nevertheless, $\mathcal{S}_{\mathrm{nl-dec}}\cap (\mathcal{E}^\mathrm{EnKU})^{\mathrm{c}}$, with ${}^\mathrm{c}$ denoting the complement in $\mathcal{F}$, is \textit{nonempty}: there exist carefully constructed distributions satisfying \eqref{eq:S_nl_dec_rec} for which exact weakly observation-dependent conditioning is possible while the EnKU is not exact. We give such an example in Subsection \ref{subsection:non_enku_example}.

\subsection{Observation-dependent gain}
The maximality result above hinges on the restriction that $A(\pi)$ and $B(\pi)$ are independent of $y_\star$. If we drop this and allow \emph{fully observation-dependent} affine maps $L_{\pi, y_\star}(x,y)=A(\pi, y_\star)x+B(\pi, y_\star)y+c(\pi, y_\star)$, 
the situation changes: one can engineer many non-Gaussian $\pi$ with exact affine transports that lie strictly beyond $\mathcal{E}^{\mathrm{EnKU}}$. We present the following example.  
\begin{example} Consider any distribution $\eta \in \Pp(\R)$ 
and measurable function $f:\R\rightarrow \R$. Define $\pi$  by pushing forward through $\phi: \R^2 \rightarrow \R^2, \phi(z,y) = (f(y)z,  y)$: 
$$
\qquad \pi_{XY} = \phi_\sharp (\eta\otimes \eta). 
$$
Clearly $\pi$ is not in $\mathcal{F}$ for general $f$ and has the exact affine conditioning map $L_{y_\star}(x,y) = f(y_\star)y$.
\end{example}
Another example is as follows.
\begin{example}
Consider the hypercube $C = [0,1]^{2}$ and any orthogonal $R \in O(2)$. Let 
$(X, Y) \sim \mathrm{Unif}\left(RC\right)$ be uniformly distributed. For any $y^\star$ in the support of $Y$ there are $a(y^\star)$, $b(y^\star)$ such that
$$
{X|Y=y^\star}  \sim \mathrm{Unif}([a(y^\star),b(y^\star)]).
$$
Thus, an exact affine conditioning map is, for example,
$$
L_{y^\star}(x, y) = (b(y^\star) - a(y^\star))e_1^\top R^\top(x, y)  + a(y^\star).
$$
\end{example}
\begin{remark}
\label{remark:learned_ensemble_filters}
This perspective aligns with recent ``learned ensemble filters'' \cite{mccabe2021learning, revach2022kalmannet, bach2025learning}, where the analysis maps are chosen as  $L_{\pi, y_\star}(x,y) =x + B(\pi, y_\star) y+ c(\pi, y_\star)$, with the gain terms  $B$ and $c$  parameterized by a neural network in an observation-dependent manner. A complementary practical example of choosing $B$ in an observation-dependent way to
enlarge the exactness domain---here to include Student-$t$ distributions---is given in
\cite{provost2023adaptive}. Our negative result in Theorem \ref{thm:F_char} for weakly observation-dependent maps helps understand why these methods pursue observation-dependent updates: without such dependence, there is essentially no headroom beyond EnKU, whereas allowing $B$ to depend on $y_\star$ (and of course on $\pi$)  
can potentially realize exact updates for broader constructions. An interesting direction is to understand the enlargement of the exactness class, relative to $\mathcal{F}$, when \(A\) and \(B\) are allowed to depend on \(y_{\star}\).
\end{remark}

\section{Numerical experiments}
\label{section:num_exp}
This section illustrates the two main points of our paper empirically.  First, in Subsection \ref{subsection:within_enku}, we demonstrate that the EnKU remains bias-free for
highly non-Gaussian and structurally complex joint distributions, whereas other
affine updates that are exact only at the level of first and second moments fail for
these examples. This behavior supports the uniqueness result of
Theorem \ref{thm:enkf_update_unique} (and Corollary~\ref{cor:enkf_char}) within the class $\mathcal{E}^{\mathrm{EnKU}}$. Second, in Subsection~\ref{subsection:non_enku_example}, we present a carefully constructed example from the
highly structured class $\mathcal{S}_{\mathrm{nl\text{-}dec}}$ appearing in Theorem~\ref{thm:F_char}, for which an exact weakly observation-dependent affine conditioning map exists while the EnKU is inexact. This example illustrates our maximality result by showing that failures of the EnKU can occur for such pathological elements of $\mathcal{F}$.

\subsection{Within $\mathcal{E}^{\mathrm{EnKU}}$} 
\label{subsection:within_enku}
We test our main claim in Theorem \ref{thm:enkf_update_unique} empirically: in the mean-field limit and within affine conditioning maps, the EnKU is the only method that remains exact beyond highly symmetric distributions (such as Gaussians) within $\mathcal{E}^{\mathrm{EnKU}}$. To also account for finite sample effects, we consider the likelihood-free Bayesian inference problem
\begin{equation}
\label{eq:likelihood_free_bayesian_inversion}
\text{given i.i.d.\ }\{(X_i,Y_i)\}_{i = 1}^{N}\sim \pi \text{ compute } \{Z_i\}_{i = 1}^N \text{ such that } \frac{1}{N}\sum\limits_{i =1}^N \delta_{Z_i }\approx\pi_{X|Y=y_\star}.
\end{equation}
For this task, we compare several affine conditioning maps $L$ to the EnKU, producing samples of the form
\begin{equation}
\label{eq:posterior_pred_affine}
Z_i = L_{\widehat{\pi}^N, y_\star}(X_i, Y_i),
\end{equation}
where $\widehat{\pi}^N$ denotes the empirical distribution of the joint samples $\{(X_i,Y_i)\}_{i=1}^N$. We study the resulting approximations as the ensemble size $N$ increases.  Specifically, we pick three joint distributions $\pi\in \mathcal{E}^{\mathrm{EnKU}}$ in dimension $n=m=2$, defining them through   a distribution $\nu\in \Pp_2\left(\R^n\right)$  and a linear map $O(z,y)= z+\begin{psmallmatrix}10 &100 \\ 0 & 1
\end{psmallmatrix}y$ just as in Equation \eqref{eq:exact_set_kalman}. Thus, $\pi$ is fully defined by the marginal choices for $Z\sim \nu$ and $Y\sim\pi_Y$ (listed below), while preserving the linear coupling that places each example in $\mathcal{E}^{\mathrm{EnKU}}$.  We test the following three instances of $(X,Y) \sim \pi$, intentionally constructed to be challenging and ill-conditioned in order to illustrate our theory.
\begin{itemize}
\item {Experiment 1 (Gaussian):} As a sanity check, we consider the standard linear-Gaussian problem that classically motivates ensemble filters, namely
\[
Z\sim\mathcal N(0,\Sigma_Z),\qquad Y\sim\mathcal N(0,\Sigma_Y).
\]
As mentioned in the introduction, infinitely many affine transports result in exact conditioning for Gaussians in the mean-field. We  use 
\begin{align*}
              \Sigma_Z = \begin{pmatrix} 10 & -2.5 \\ -2.5 & 1 \end{pmatrix},\quad \Sigma_Y = \begin{pmatrix} 1 & 1.5 \\ 1.5 & 19\end{pmatrix}.
            \end{align*}
\item {Experiment 2 (Gaussian mixtures):} This is an example of a distribution that is in the set $\mathcal{E}^\mathrm{EnKU}$ but strongly multimodal and non-Gaussian: 
\[
Z \sim \sum_{k=1}^{6} w^{(Z)}_k\mathcal N(\mu^{(Z)}_k,\Sigma^{(Z)}_k),\qquad
Y \sim \sum_{\ell=1}^{6} w^{(Y)}_\ell\mathcal N(\mu^{(Y)}_\ell,\Sigma^{(Y)}_\ell). 
\]
The parameters $w_\ell$, $\mu_\ell$, and $\Sigma_\ell$ are drawn independently  from 
\begin{align*}
w^{(Z)}, w^{(Y)} &\stackrel{\text{i.i.d.}}{\sim} \mathrm{Dir}\left(\mathbf{1}_{6}\right) \\
\mu^{(Z)}_k,  \mu^{(Y)}_k &\stackrel{\text{i.i.d.}}{\sim} \mathcal{N}(0, S)& k = 1,\ldots, 6  \\
\Sigma_k^{(Z)},  \Sigma_k^{(Y)} &\stackrel{\text{i.i.d.}}{\sim} \mathcal{C}  & k = 1,\ldots, 6 
\end{align*}
with $\mathrm{Dir}$ denoting the Dirichlet distribution, $\mathbf{1}_{6}$ the vector of $6$ ones, covariance matrix 
$$
S =\begin{pmatrix}
120 & 108\\
108 & 120
\end{pmatrix},
$$
and $\mathcal{C}$ defined as the law of the matrix $M$ in 
\begin{align*}
F &\in \R^{2\times 2}, (F)_{ij} \stackrel{\text{i.i.d.}}{\sim} \mathcal{N}(0,1),
\quad s \in \R^2, 
s_i \stackrel{\text{i.i.d.}}{\sim} \text{Unif}(0.2,1.5) \\
M &= F \mathrm{diag}(s)F ^\top + 10^{-6} I_2. 
\end{align*}
\item {Experiment 3 (elliptical ring density):} We consider another strongly non-Gaussian distribution that is in $\mathcal{E}^\mathrm{EnKU}$.  Consider $K = 3$ rings and $M=6$ angular modes. We spread out the radii $\ell_r, r = 1,..., K$ uniformly between $\ell_1 = 1.4$ and $\ell_K = 4.0$. Consider an independently uniformly distributed ring mode $r \sim \mathrm{Unif}(\{1, \ldots, K\})$ and angular mode $j \sim \mathrm{Unif}(\{1, \ldots, M\})$ with centers $\mu_j=\tfrac{2\pi j}{M}$. Conditioning on $(r,j)$, let
\[
\theta \mid j \sim \mathrm{vM}(\mu_j,\kappa), \kappa=25,\quad
\rho \mid r \sim \mathcal N(\ell_r,\sigma^2)
\]
with $\mathrm{vM}$ denoting the von Mises distribution and $\sigma = 0.2$. This defines the random variable $U$  with ring density through the polar parametrization
$$
U = \begin{pmatrix}\rho\cos\theta\\ \rho\sin\theta\end{pmatrix}.
$$
We break Euclidean $6$-fold rotational symmetry by defining
\[
Z := \begin{pmatrix}
2 & 0\\
1 & 1
\end{pmatrix} U.
\]
For $Y$, we consider a Gaussian mixture with $6$ components, sampled as in Experiment~2, with the modification that $
\mu^{(Y)}_k \stackrel{\text{i.i.d.}}{\sim} \mathcal{N}\left(0,
\begin{psmallmatrix}
150 & 0 \\
0 & 1
\end{psmallmatrix}\right), k = 1,\ldots,6.$
\end{itemize}

\begin{figure}[htbp]
  \begin{minipage}{0.49\textwidth}
    \includegraphics[width=\linewidth,clip,trim=0 0 0 0]{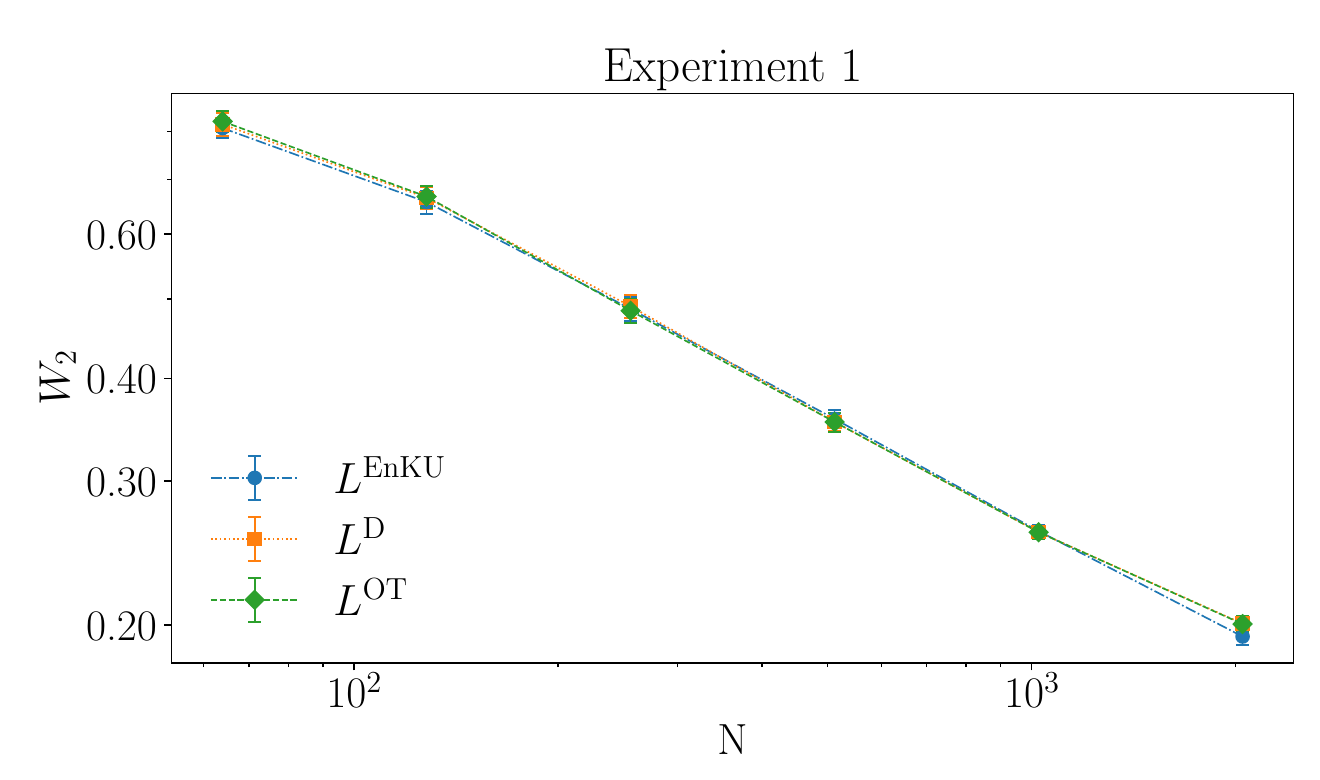}
  \end{minipage}\hfill
  \begin{minipage}{0.49\textwidth}
    \includegraphics[width=\linewidth,clip,trim=0 0 0 0]{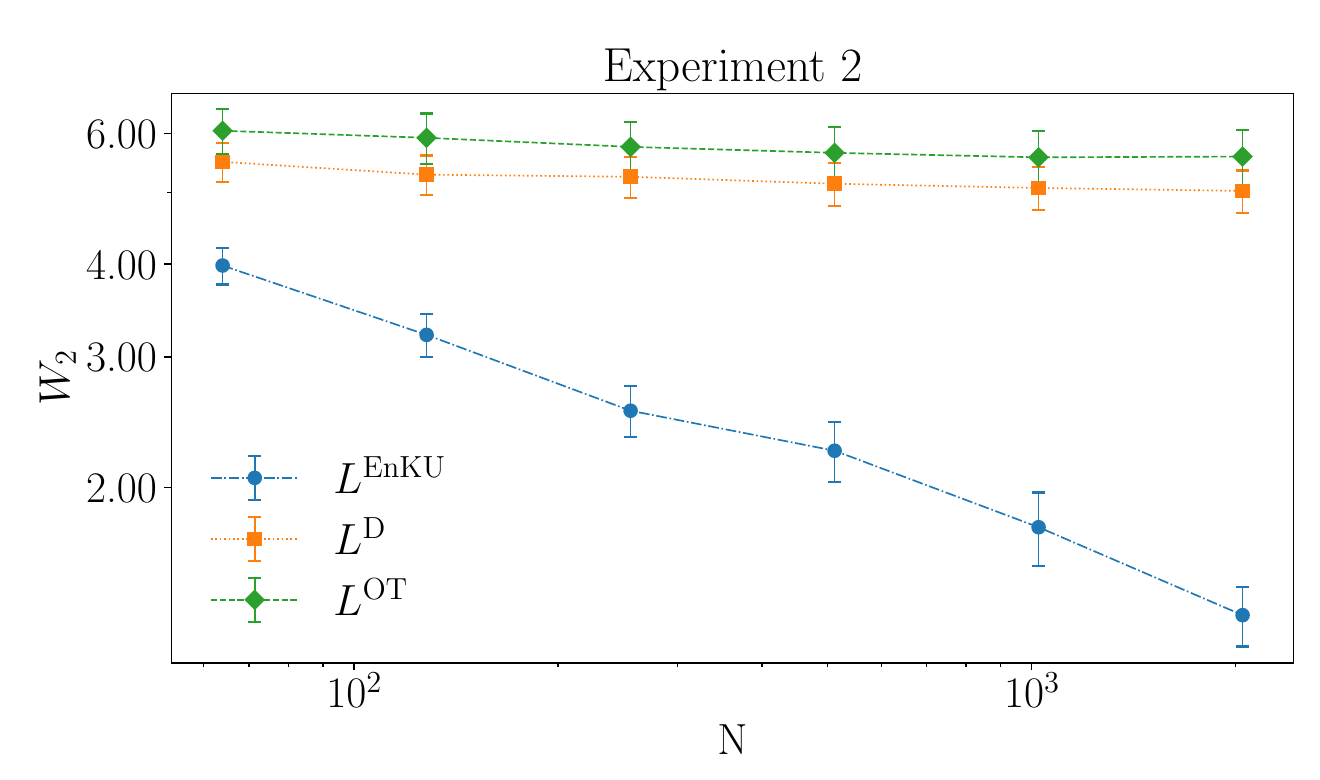}
  \end{minipage}\\
  \begin{minipage}{0.49\textwidth}
    \includegraphics[width=\linewidth,clip,trim=0 0 0 0]{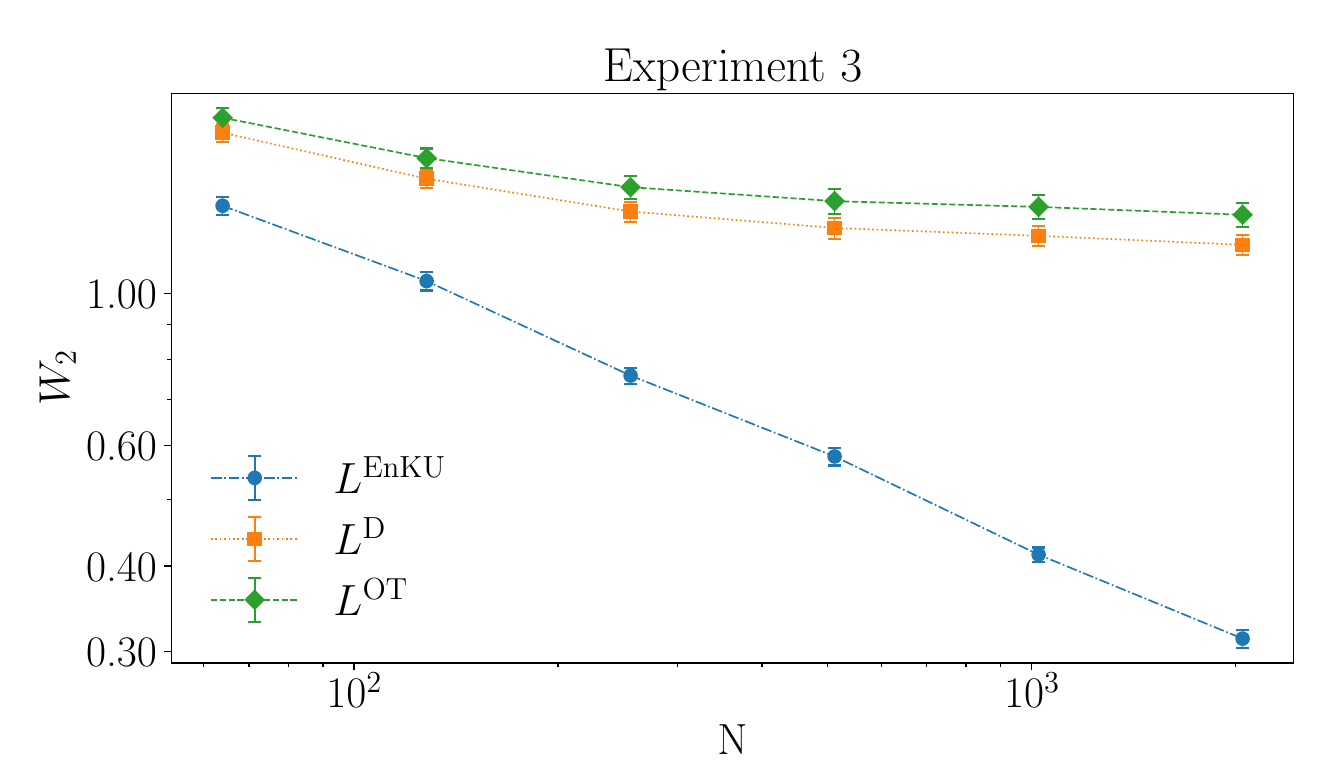}
  \end{minipage}\hfill
\caption{{Convergence of affine updates with ensemble size.} 
Log-log $W_2$ error versus ensemble size $N$ for the three data-generating models.
{Experiment 1 (Gaussian):} all Gaussian-exact affine maps exhibit decreasing error with $N$ (no bias floor). 
{Experiments 2 and 3 (non-Gaussian):} EnKU continues to improve with $N$, whereas the alternative affine maps plateau at a nonzero {bias floor}, indicating mean-field bias under non-Gaussian structure. Error bars show mean~$\pm$~standard error\ over $40$ Monte Carlo replicates. For Experiment 2, this also includes resampling of the model parameters. 
}
\label{fig:w2_vs_err}
\end{figure}

\medskip
For the posterior approximation in Equation \eqref{eq:posterior_pred_affine}, we will compare the  EnKU, with Kalman gain $\hat K = \hat \Sigma_{XY}\hat\Sigma_{YY}^\dagger$ estimated from the sample covariances $\hat \Sigma$, to two other affine updates used in likelihood-free Bayesian inversion.  First, we will compare to the non-stochastic choice, 
\begin{equation}
\label{eq:map_D}
L_{\widehat{\pi}^N, y_\star}^\text{D}(x,y) =\sqrt{\hat \Sigma_{X|Y}} \hat\Sigma_X^{\dagger/2} (x-\hat m_X) + \hat K(y_\star - \hat m_Y) + \hat m_X,
\end{equation}
which is, for example, introduced in \cite{calvello2025ensemble}.   $\hat m_Y$ (resp.\ $\hat m_X$) is the sample mean of $Y_i$ (resp.\ $X_i$) and $\hat \Sigma_{X|Y} := \hat \Sigma_{X} - \hat \Sigma_{XY} \hat \Sigma_{YY}^\dagger\hat \Sigma_{YX}$. All square roots in the equation above are principal choices and we define
$$
M^{\dagger/2} :=\sqrt{M^\dagger}
$$
for positive semidefinite square matrices $M$. Second, we compare to another  non-stochastic affine transport, \fjedit{obtained by applying the \textit{Gaussian-to-Gaussian} optimal transport formula to the empirical first and second moments},
\begin{equation}
\label{eq:map_OT}
L_{\widehat{\pi}^N, y_\star}^\text{OT}(x,y) =\hat \Sigma_{X}^{\dagger/2}  \left(\sqrt{\hat \Sigma_{X}}   \hat \Sigma_{X|Y} \sqrt{\hat \Sigma_{X}} \right)^{1/2}\hat \Sigma_{X}^{\dagger/2}(x-\hat m_X) + \hat K(y_\star - \hat m_Y) + \hat m_X.
\end{equation}
The choices $L_{\widehat{\pi}^N, y_\star}^\text{D}$ and  $L_{\widehat{\pi}^N, y_\star}^\text{OT}$ are particular versions of  ensemble square root filters---more specifically, \textit{ensemble adjustment Kalman filters} (EAKF) \cite{tippett2003ensemble, whitaker2002ensemble, bishop2001adaptive, hunt2007efficient}.  This can be seen by a straightforward calculation of the mean and covariance. A fuller derivation, explaining connections to the EAKF and the ensemble transform Kalman filter (ETKF), is in Supplement~\ref{section:connection_esrf}.
Each of these three affine maps is exact for Gaussian distributions at the mean-field level, \fjedit{meaning all three methods predict the same correct Gaussian posterior, and accordingly we expect all methods to recover the correct posterior in Experiment~1.}
We condition on the fixed observation $y_\star=(0.4,-0.2)^\top$.   Ignoring finite sample effects, the particular value of \(y_\star\) plays no role in the relative performance of the methods, since all updates depend on \(y_\star\) exclusively through the same affine shift \(\hat K y_\star\).  We run the affine ensemble algorithms at increasing ensemble sizes, investigating the $W_2$-error of their predicted posterior compared to the true posterior.  In Figure \ref{fig:w2_vs_err},  we estimate the empirical $W_2$ between the predicted analysis ensembles and i.i.d.\ samples from the ground-truth posterior using  POT's \texttt{ot.emd2} algorithm for each ensemble size $N$.  
The results match the mean-field predictions. In {Experiment 1} (Gaussian), all three affine maps show error decreasing parametrically with $N$ and no bias floor. For distributions in $\mathcal{E}^{\mathrm{EnKU}}$ that are non-Gaussian ({Experiments 2 and 3}), the EnKU continues to improve as $N$ grows, while the alternative affine maps stabilize at a nonzero error, revealing a mean-field {bias floor}. The posterior density plots in Figure \ref{fig:densities} demonstrate this further: the EnKU reproduces the multimodal and ring-like posterior structure, whereas the other affine updates smear or collapse features, consistent with their moment-matching but distributionally biased behavior. 
\begin{figure}[!tbp]
\includegraphics[width=\linewidth,clip,trim=0 0 0 0]{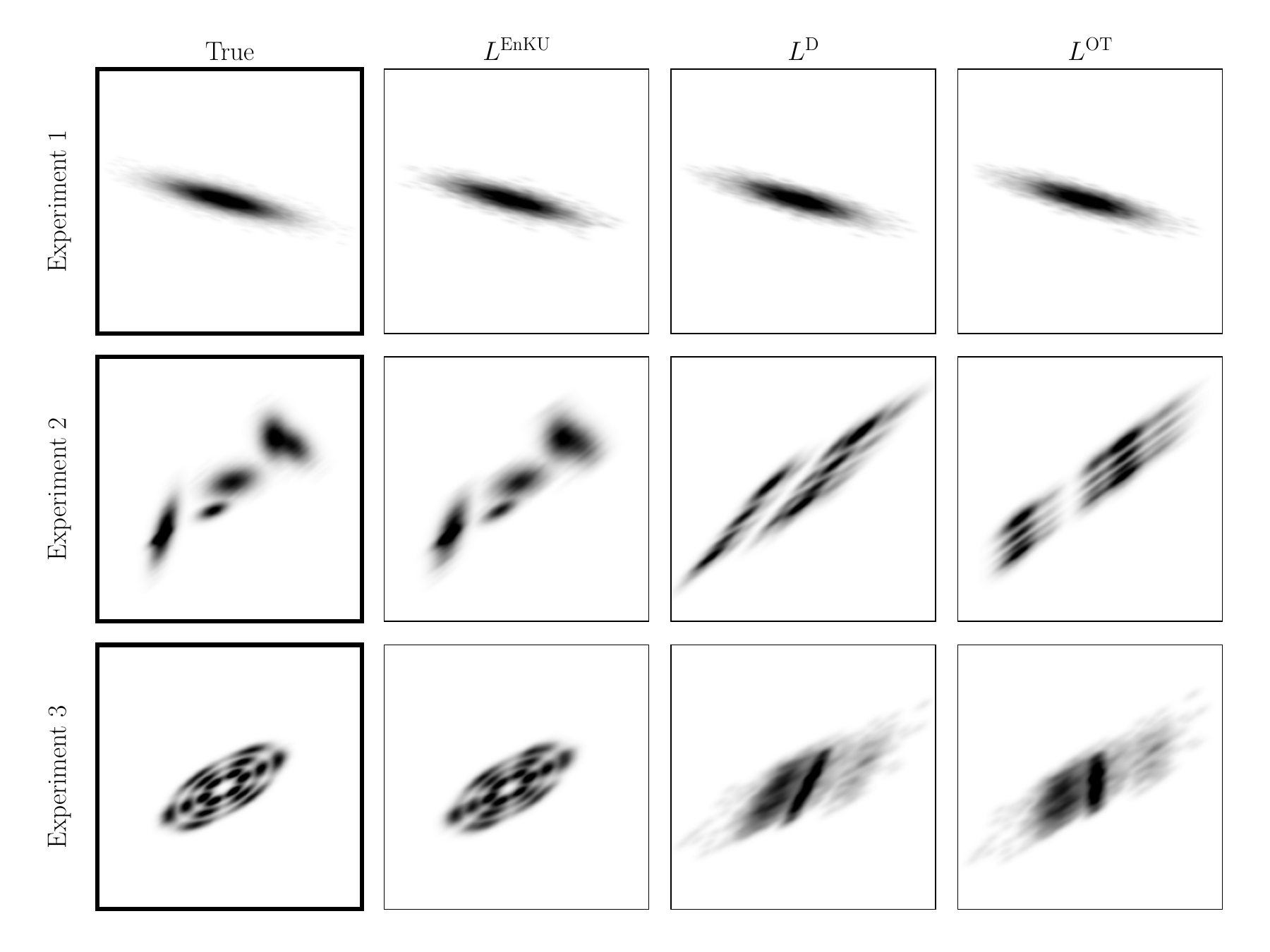}
\caption{
Posterior structure recovered by each method for $N = 8{,}192$ samples ($49{,}152$ samples for the true posterior).
For each experiment, we show i.i.d.\ samples from the true posterior  alongside posterior ensembles produced by EnKU, the deterministic map $L^{\mathrm D}$, and the OT map $L^{\mathrm{OT}}$. Densities are visualized using a two--dimensional Gaussian kernel density estimator with Scott’s rule and a reduced bandwidth factor of \(0.6\). 
In the Gaussian case (Experiment 1), all methods match the target shape. In the non-Gaussian cases (Experiments 2 and 3), the EnKU best preserves multimodality and ring structure, while $L^{\mathrm D}$ and $L^{\mathrm{OT}}$ blur or collapse features--visual evidence of the {bias floor} quantified in the $W_2$ plots.}
  \label{fig:densities}
\end{figure}
\begin{remark}
\fjedit{The comparisons above test whether each population-level affine map incurs a nonzero bias.
A practical filter must additionally control finite-ensemble variance, especially important for small ensembles and in high dimensions. Deterministic square-root variants avoid sampling error introduced by stochastic observations, and can be advantageous in this regard \cite{whitaker2002ensemble, tippett2003ensemble, lawson2004implications}. 
While finite-ensemble variance is not a central concern of this paper, \ifarxiv Appendix~\ref{section:finite_sample_effects}\else Supplement~\ref{section:finite_sample_effects}\fi \ reports an experiment assessing it for the examples of this section.}
\end{remark}
\subsection{Beyond  $\mathcal{E}^{\mathrm{EnKU}}$}  
\label{subsection:non_enku_example}
In order to illustrate the class $\mathcal{S}_{\mathrm{nl\text{-}dec}}$  introduced in Theorem \ref{thm:F_char}, we present a simple scalar ($n=m=1$) example $\pi$  that lies in the nonlinear decomposability class $\mathcal{S}_{\mathrm{nl\text{-}dec}}$ but not in $\mathcal{E}^{\mathrm{EnKU}}$, for which an exact {weakly observation-dependent} affine conditioning map exists, while the EnKU is biased. Let $Y\sim\mathcal N(0,1)$ and define the nonlinear shift 
$d(y) := y^2 - 1.$ 
Fix $\lambda\in(0,1)$ and $B\in\mathbb R\setminus\{0\}$. Let $\{Y^{(k)}\}_{k\geq 0}$ be i.i.d.\ copies of $Y$, independent of $Y$, and define
\begin{equation}
\label{eq:nldec_Z_def}
Z
:= \sum_{k=0}^{\infty} \lambda^k\bigl(BY^{(k)} + \lambda d(Y^{(k)})\bigr)
\end{equation}
according to Equation \eqref{eq:S_nl_dec_rec}, which converges in $L^2$ since $\lambda\in(0,1)$ and $Y$, $d(Y)$ have finite second moments. Define the joint distribution $\pi$ of $(X,Y)$ by $X := Z + d(Y). $ 
An exact weakly observation-dependent affine conditioning map is given by 
$$
L_{\pi,y_\star}^{\mathrm{nl\text{-}dec}}(x,y)
:= \lambda x + B y + d(y_\star)
$$
as the following computation shows:
\begin{align*}
    L_{\pi,y_\star}^{\mathrm{nl\text{-}dec}}(X,Y) &= \lambda Z + \lambda d(Y) + B Y + d(y_\star) \\
   & \stackrel{d}{=} Z + d(y_\star)\\
&\stackrel{d}{=} X\mid Y=y_\star.
\end{align*}
The second equality above follows from the fixed-point-in-distribution identity $Z \stackrel{d}{=} \lambda Z + \lambda d\bigl(Y\bigr) + B Y,$ which is an immediate consequence of the definition in Equation \eqref{eq:nldec_Z_def}.  In particular, we have  $\pi\in\mathcal F$.  Further, since $\mathbb E(X\mid Y=y)=d(y)$ is nonlinear in $y$, there does not exist a representation
$X = Z_0 + MY$ with $Z_0 \indep Y$ and $M\in\mathbb R$. Consequently, $\pi \notin \mathcal{E}^{\mathrm{EnKU}}$ by Proposition \ref{prop:exact_set_kalman} and the EnKU is inherently biased for this distribution. Hence  $\pi \in \mathcal{S}_{\mathrm{nl\text{-}dec}}$.  
Indeed, in Figure \ref{fig:posterior_nl_dec} we compare the posterior distributions produced by the EnKU and by
$L_{\pi,y_\star}^{\mathrm{nl\text{-}dec}}$ according to the updating Equation \eqref{eq:posterior_pred_affine} in finite samples. 
While $L_{\pi,y_\star}^{\mathrm{nl\text{-}dec}}$ exactly recovers the true posterior, the EnKU is inexact for this distribution and exhibits strong bias. This illustrates that weakly observation-dependent affine conditioning maps can achieve exactness beyond the EnKU domain, but only for specially structured distributions.
\begin{figure}[!tbp]
\centering
\includegraphics[width=\linewidth,clip,trim=0 0 0 0]{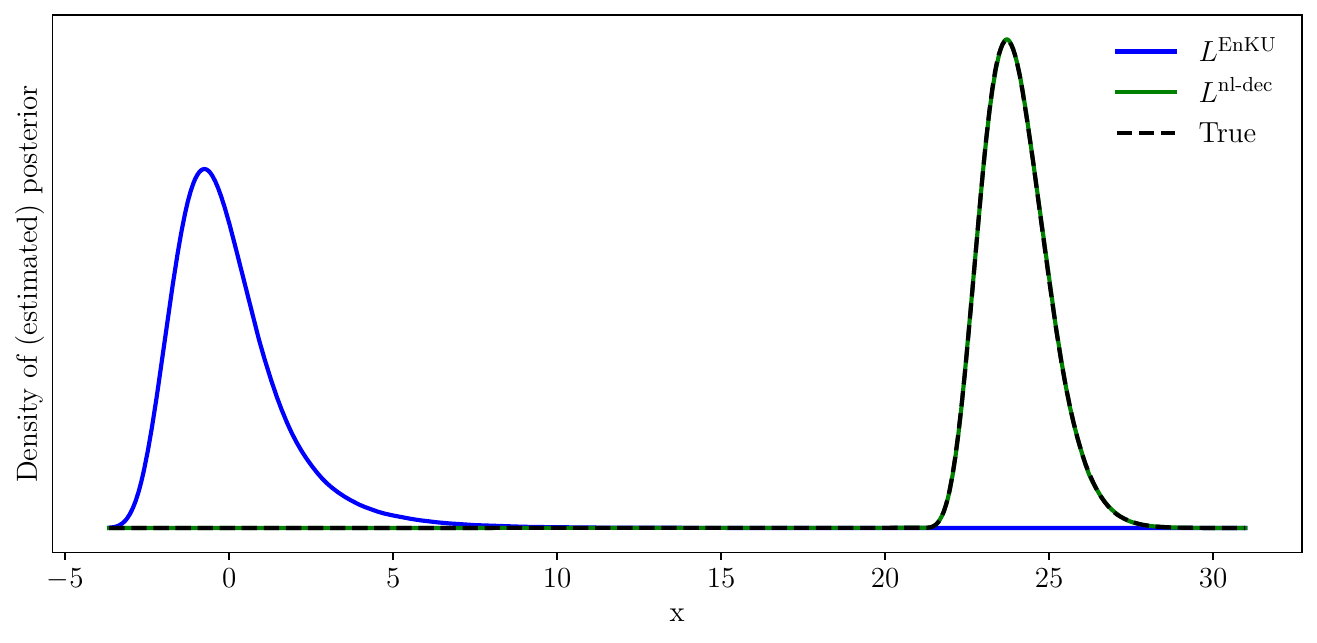}
\caption{Posterior at $y_\star=5$ produced by $L_{\pi,y_\star}^{\mathrm{nl\text{-}dec}}$ and
$L_{\pi,y_\star}^{\mathrm{EnKU}}$ using $N=10^6$ samples as in Equation \eqref{eq:posterior_pred_affine}, with parameters
$\lambda=0.1$ and $B=1$. Densities are estimated using kernel density estimation with Scott’s rule and a reduced bandwidth factor of $1.2$. 
In agreement with our theoretical analysis, the exact weakly observation-dependent map $L_{\pi,y_\star}^{\mathrm{nl\text{-}dec}}$ coincides with the true posterior, while the EnKU update is substantially biased, illustrating that exact affine conditioning beyond $\mathcal{E}^{\mathrm{EnKU}}$ can occur for highly structured nonlinear
distributions. }
  \label{fig:posterior_nl_dec}
\end{figure}
Importantly, the example presented here is quite non-generic: it relies on the strong nonlinear decomposability condition in Equation \eqref{eq:S_nl_dec_rec} defining $\mathcal{S}_{\mathrm{nl\text{-}dec}}$, which severely restricts the class of admissible
joint distributions. Accordingly, the purpose of this example is \emph{not} to advocate the use of the map $L_{\pi,y_\star}^{\mathrm{nl\text{-}dec}}$ in practice, but rather to illustrate that departing from $\mathcal{E}^{\mathrm{EnKU}}$  in Theorem~\ref{thm:F_char} requires detailed {a priori} knowledge of the joint distribution that is generally unavailable.

\section{Discussion} 
Our maximality result for weakly observation-dependent affine maps shows that there is essentially no headroom beyond the EnKU: the largest possible exact set $\mathcal{F}$ collapses to $\mathcal{E}^{\mathrm{EnKU}}$ up to the narrow symmetry class $\mathcal{S}_{\mathrm{nl-dec}}$ (Theorem \ref{thm:F_char}). Further, we show that within $\mathcal{E}^{\mathrm{EnKU}}$, the EnKU is the unique affine  exact conditioning map up to small symmetry classes $\mathcal{S}_{\mathrm{cov}}$, $\mathcal{S}_{\mathrm{dec}}$, and $\mathcal{S}_{\mathrm{cyc}}$ (Theorem \ref{thm:enkf_update_unique}).  

Many questions remain open. Importantly, our analysis is mean-field and {does not address} many practical issues---finite-$N$ sampling error, localization, covariance inflation, model error/misspecification, and adaptive tuning schemes---which are known to strongly impact performance. \fjedit{Relatedly, our analysis concerns only affine \emph{transport} maps, which act on each ensemble member individually, and not ensemble-space \emph{transforms} that build each posterior ensemble member as a superposition (linear combination) of the forecast members, as in the ETKF.} Further, in practical data assimilation and inverse problems, the true joint distribution rarely lies in $\mathcal{E}^{\mathrm{EnKU}}$ and deviates even further from Gaussianity. Regardless, affine filters are applied in these settings. Therefore, another lens for studying the choice of affine filters is the bias--variance tradeoff. Affine filters are usually used in high dimensions where the dimension is large compared to the ensemble size, and where the ensemble is non-i.i.d.\ after one filtering step. For these two reasons, accepting bias in the estimator to reduce variance is inevitable. Quantifying this tradeoff in nonlinear settings remains an important direction. 
A related open question is to treat the corresponding multi-step behavior of the EnKU (e.g., in ensemble Kalman inversion) and to understand how nonlinear effects re-enter through evolving covariances \cite{iglesias2013ensemble, schillings2017analysis,schillings2018convergence}.

\appendix

\ifarxiv
\section{Connection to ensemble square root filters}
\label{section:connection_esrf}
\input{connection_square_root}

\section{\fjedit{Finite-sample effects}}
\label{section:finite_sample_effects}
\fjedit{\input{finite_sample_effects}}

\section{\fjedit{Additive independent-noise models}}
\label{section:additive_noise_models}
\fjedit{\input{notes_inverse_problems}}

\subsection{\fjedit{Bias outside the exact set: a numerical example}}
\label{section:bias_outside_exact_set}
\fjedit{\input{bias_outside_exact_set}}

\section{Meagreness of the Symmetry Classes}
\label{sec:small_symmetry_topological}
\input{meagreness}
\fi

\section{Proofs}
We start by presenting a proof of Proposition \ref{prop:exact_set_kalman}
\begin{proof}[Proof of Proposition \ref{prop:exact_set_kalman}]
$\subseteq:$ Pick $\pi \in \mathcal{E}^\mathrm{EnKU}$, meaning that 
$$ 
\left(L^{\mathrm{EnKU}}_{\pi, y_\star}\right)_\sharp\pi = \pi_{X|Y=y_\star}
$$
$\pi_Y$-a.s.\ in $y_\star\in\R^m$. Letting $(X,Y) \sim \pi$ and defining $Z = X - KY$,   the equation above is equivalent to 
$$
\Law\left(Z + K y_\star \right) = \pi_{X|Y=y_\star}. 
$$
Since $Z$ does not depend on $y_\star$, this shows that for $\nu= \Law(Z)$ and $O(x,y) = x  +Ky$ we have 
$$
O(\cdot,y_\star)_\sharp \nu = \pi_{X|Y=y_\star}
$$
$\pi_Y$-a.s.\ in $y_\star$.  $O(\cdot,y_\star)_\sharp \nu$ is a Markov kernel, concluding this direction.

\medskip
$\supseteq$: Consider $\pi$  and its corresponding $O$ and $\nu$ as in the right-hand side  of the equation we prove in this proposition. Write    \( O(x, y) = A_1 x + A_2 y \) for matrices \( A_1 \in \mathbb{R}^{n \times n} \), \( A_2 \in \mathbb{R}^{n \times m} \), and let \( Z \sim \nu \). Then $\pi_{X \mid Y = y} = \Law(A_1 Z + A_2 y).$ Let \( Y \sim \pi_Y \), independent of \( Z \), so that  $(X, Y) := (A_1 Z + A_2 Y, Y) \sim \pi$. By direct computation,  $$\Sigma_{XY}(\pi) = \Cov(A_1 Z + A_2 Y, Y) = A_2 \Cov(Y) = A_2 \Sigma_{YY}(\pi).$$ Therefore,  $L^\mathrm{EnKU}_{\pi,y^\star}(x, y) = x + A_2 \Sigma_{YY}(\pi) \Sigma_{YY}^\dagger(\pi) (y^\star - y).$
Let \( \tilde Y \sim \pi_Y \), independent of \( (Z, Y) \). Define the projection \( P_Y := \Sigma_{YY}(\pi) \Sigma_{YY}^\dagger(\pi) \), the orthogonal projection onto \( \operatorname{Im}(\Sigma_{YY}(\pi)) \). Then $L^\mathrm{EnKU}_{\tilde Y}(A_1 Z + A_2 Y, Y) = A_1 Z + A_2 Y + A_2 P_Y (\tilde Y - Y)$. 
Because \( Y - \tilde Y \in \operatorname{Im}(\Sigma_{YY}(\pi)) \) a.s., this a.s.\ simplifies to $L^\mathrm{EnKU}_{\tilde Y}(A_1 Z + A_2 Y, Y) = A_1 Z + A_2 \tilde Y$.
Thus,
\[
\Law\left(L^\mathrm{EnKU}_{\tilde Y}(X, Y)\left|\tilde Y\right.\right) = \Law(A_1 Z + A_2 \tilde Y) = \pi_{X \mid Y = \tilde Y}
\]
concluding the proof. 
\end{proof}
Similarly, we can show Proposition \ref{prop:char_weak_set}. 
\begin{proof}[Proof of Proposition \ref{prop:char_weak_set}]
Let $\pi\in \mathcal{F}$. Then there exists a weakly observation-dependent affine map
\[
L_{\pi,y_\star}(x,y)=A(\pi)x+B(\pi)y+c(\pi,y_\star)
\]
and a Markov kernel such that there is a Borel set $Q\in \mathcal{B}(\R^m)$ with $\pi_{Y}(Q) = 1$ and 
\[
(L_{\pi,y_\star})_\sharp \pi=\pi_{X|Y=y_\star}
\]
for all $y_\star \in Q$. Since $A$ and $B$ do not depend on $y_\star$,  for any  $y_0,y_\star \in Q$ we have 
\[
\pi_{X|Y=y_\star} = T_{c(\pi,y_\star)-c(\pi,y_0)}\nu
\]
where we set $\nu:=\pi_{X|Y=y_0}$. Now, we construct a measurable  $d(\cdot)$ such that 
$$
d(y_\star) = c(\pi,y_\star)-c(\pi,y_0)
$$
for all $y_\star \in Q$ and note that this concludes the proof. Define $d : \R^m \rightarrow \R^n$ through 
\[
d(y_\star) =
\begin{cases}
 c(\pi,y_\star)-c(\pi,y_0), & y_\star \in Q, \\[6pt]
0, & y_\star \notin Q.
\end{cases}
\]
For any Borel set $W\in \mathcal{B}(\R^n)$ we have 
$$
d^{-1}(W) = (Q\cap d^{-1}(W)) \cup (Q^{\mathrm{c}} \text{ if }0 \in W) =  d_{|Q}^{-1}(W) \cup (Q^{\mathrm{c}} \text{ if }0 \in W)
$$
meaning that all we have to show is that the restriction $d_{|Q}$ is measurable.  Consider the translation map
\[
\Phi:\R^n\rightarrow \Pp_2(\R^n),\qquad \Phi(h):=T_h\nu 
\]
where $\Pp_2(\R^n)$ is endowed with the Wasserstein-topology. The map \(\Phi\) is continuous and injective; by the Lusin--Souslin Theorem \cite[Lemma 8.3.8]{cohn2013measure} and since $\R^n$ and $\Pp_2(\R^n)$ are Polish spaces, the  inverse on its image $\mathcal{O} = \Phi(\R^n)$, namely
\[
\Psi:\mathcal{O}\rightarrow \R^n,\qquad \Psi(T_h\nu)=h,
\]
is measurable with respect to the Borel algebra induced by the subspace topology of $\mathcal{O}$. By the first part of the proof in \cite[Lemma 12.4.7]{ambrosio2005gradient}, the map \(y\mapsto \pi_{X|Y=y}\) is \(\mathcal{B}(\R^m)\)-to-Borel\((\Pp_2(\R^n))\) measurable; hence its restriction \(Q\rightarrow \Pp_2(\R^n)\) is \((Q,\mathcal{B}(Q))\)-measurable. We established that on \(Q\) we have \(\pi_{X|Y=y}\in \mathcal{O}\) and thus we have that
\[
d_{|Q}(y)=\Psi\bigl(\pi_{X|Y=y}\bigr),\qquad y\in Q.
\]
and $d_{|Q}$ is measurable as a composition of measurable maps $y \mapsto \pi_{X|Y=y} \mapsto \Psi(\pi_{X|Y=y})$.  $T_{d(y_\star)}\nu$ is a valid choice of Markov kernel by measurability of $d$. Further, $d$ is $\pi_Y$-a.s.\ unique by $\pi_Y$-a.s.\ uniqueness of Markov kernels. 
\end{proof}
The following theorem, while not explicitly stated in the main paper, is the main theoretical basis for the remaining results presented in this paper. 
\begin{theorem}
\label{thm:main_result}
Let $A\in\mathbb{R}^{n\times n}$  and let $U$ be an $\mathbb{R}^n$-valued random vector with $\mathbb{E}\|U\|^2<\infty$. Assume $X\in\R^n$ is independent of $U$. Consider the fixed-point-in-distribution equation
\begin{equation}
\label{eq:fixed_point}
X \stackrel{d}{=} AX + U.
\end{equation}
By the real Jordan decomposition, there exist $A$-invariant subspaces such that 
\[
\mathbb{R}^n = V_s \oplus V_r \oplus V_u
\]
and for all restrictions $A_\bullet := A_{|V_\bullet}, \bullet \in\{u,s,r\}$
\begin{enumerate}
    \item all complex eigenvalues of $A_s$ have magnitude less than $1$
    \item all complex eigenvalues of $A_r$ have magnitude equal to $1$
    \item all complex eigenvalues of $A_u$ have magnitude larger than $1$.
\end{enumerate}
Further, decompose the complexification $V_r^\C\subseteq\C^n$
$$
 V_r^\C =  V_r^{(1)} \oplus V_r^{(2)}
$$
with $ V_r^{(1)} $ the space of all eigenvectors of $A_r$ with eigenvalues $|\lambda| = 1$. Denote by $P_\bullet$ the corresponding projections and write $X_\bullet:=P_\bullet X$, $U_\bullet:=P_\bullet U$. There exists a solution $X$ with $\mathbb{E}\|X\|_2^2<\infty$ to Equation \eqref{eq:fixed_point} if and only if $U_u$ and $U_r$  are a.s.\ constant vectors and $U_r \in \operatorname{Im}(I - A_r)$ a.s.\ The blockwise solutions, if they exist, satisfy:
\begin{itemize}
\item[\emph{(a)}] There is a \emph{unique} solution in distribution in the stable component given by
\[
X_s \stackrel{d}{=} \sum_{k=0}^{\infty} A_s^{k} U^{(k)}_s,
\]
where $\{U^{(k)}_s\}_{k\ge0}$ are i.i.d.\ copies of $U_s$, independent of each other; the series converges in $L^2$.
\item[\emph{(b)}]  $X_r^{(2)}$ is a.s.\ constant. 
\item[\emph{(c)}] $X_u$ is a.s.\ constant with the a.s.\ value 
\[
X_u = (I-A_u)^{-1} U_u. 
\]
\end{itemize}
\end{theorem}
Before presenting a proof, we need to show a few lemmas. 
\begin{lemma}
\label{lem:B_op_norm_bound}
Consider a matrix $B\in\R^{d\times d}$ with $\rho(B)< 1$. Then there is a norm $\| \cdot \|$ on $\R^d$ such that the operator norm satisfies $\|B\| <1$.    
\end{lemma}
\begin{proof}
The discrete Lyapunov equation
\[
B^\top P B - P = -I
\]
has a unique positive-definite solution $P\succ0$ \cite{parks1964liapunov}. Define the (equivalent) norm $\|x\|_P := (x^\top P x)^{1/2}$. Then
\[
\|B x\|_P^2
= x^\top B^\top P B x
= x^\top (P-I)x
= \|x\|_P^2 - \|x\|_{I}^2
\le \Big(1-\tfrac{1}{\lambda_{\max}(P)}\Big)\|x\|_P^2.
\]
Note that $\|x\|_P^2 = x^\top P x= x^\top B^\top P B x + \|x\|_I^2 \geq \|x\|_I^2$ implies that $\lambda_{\max}(P)\geq 1$.  Hence $\|B\|_P \le \sqrt{1-1/\lambda_{\max}(P)}=:q<1$ as claimed.
\end{proof}
\begin{lemma}\label{lem:stein-noncontractive}
Let $J\in\mathbb R^{n\times n}$ be a Jordan block for an eigenvalue $|\lambda| = 1$. Let $Q\succeq 0$. If a symmetric $P\succeq 0$ satisfies the discrete Lyapunov (Stein) equation
\[
P = J P J^* + Q,
\]
then necessarily $Q=0$. Further, all entries of $P$ but $P_{11}$ must be zero. If instead $|\lambda| > 1$, there can only be a solution if $P = Q = 0$. 
\end{lemma}
\begin{proof}
Consider the case $|\lambda| = 1$ first. If $n = 1$ there is nothing to show, so we can assume $n > 1$.  Say $P\succeq0$ is a matrix satisfying the Lyapunov equation.  Write $J = \lambda I+ N $. Translating the Lyapunov equation $P =(\lambda I+ N ) P (\overline\lambda I+ N^T ) + Q$ into components, writing $p_{ij}$ and $q_{ij}$ for the entries of $P$ and $Q$   yields 
$$
\lambda p_{i, j+1}+\bar{\lambda} p_{i+1, j}+p_{i+1, j+1}+q_{i j}=0
$$
for all $i, j$ with \(p_{a b}=0\) if an index exceeds \(n\). We proceed by induction over $n+1\geq m > 1$ with the hypothesis that $q_{m,m} = p_{m,m} = 0$. Our inductive base is $m = n+1$ for which there is nothing to show. Let $m > 1$ and assume that $p_{m+1,m+1} = q_{m+1,m+1} = 0$. Then by Cauchy-Schwarz also $p_{ij} = q_{ij} = 0$ if either $i$ or $j$ is $m+1$. The $(i,j) = (m,m)$ equation tells us that $q_{m,m} = 0$. The $(i, j) = (m,m-1)$ equation says $\lambda p_{m,m} + q_{m,m-1} = 0$ and by Cauchy-Schwarz $q_{m,m-1} = 0$, showing that $p_{m,m} = q_{m,m} = 0$. This induction shows that $q_{ij} = p_{ij} = 0$ except for $i = j = 1$. Finally, the $(i,j)=(1,1)$ equation is  simply $q_{i j}=0$, completing the proof of the first part. 

Now let $|\lambda| >1$. Our strategy is to construct a unique solution to the unconstrained problem for $P$ and show uniqueness for this solution.  Consider the series
\[
P = -\sum_{k=1}^\infty J^{-k} Q J^{-*k}.
\]
Because $\rho(J^{-1})<1$, by Lemma \ref{lem:B_op_norm_bound} the sequence $\|J^{-k}\|$ decays geometrically in some matrix norm, ensuring absolute convergence of the series. A direct computation shows that it solves the unconstrained equation
\[
J P J^* = -\sum_{k=1}^\infty J^{-(k-1)} Q J^{-* (k-1)}
            = -Q - \sum_{k=1}^\infty J^{-k} Q J^{-*k}.
\]
Say $Q\neq 0$. Then $P$ is negative semidefinite and nonzero contradicting positive semidefiniteness. Therefore $Q = P = 0$ for this solution. We conclude the proof by showing that this is the unique solution of the unconstrained problem. The unconstrained problem is a linear operator problem that can be vectorized
$$
\Psi(X) = \mathrm{vec}(Q)
$$
with $\Psi(X) = \mathrm{vec}(X) -(J^*\otimes J)\mathrm{vec}(X)$.  The spectrum of $J^*\otimes J$ consists of the products $\{\bar\lambda_i \lambda_j\}$ where $\{\lambda_i\}$ are the eigenvalues of $J$. Since every $|\bar\lambda_i \lambda_j|>1$, the operator $\Psi$ has a trivial kernel and the solution is unique.
\end{proof}
The following is a well-known fact following from the equidistribution theorem \cite{weyl1916gleichverteilung}.
\begin{lemma}
\label{lem:dense_in_unit_int}
Let $r \in \mathbb{R}$. The set 
\[
\{ nr \bmod 1 : n \in \mathbb{Z}\}
\]
is dense in $[0,1]$ if and only if $r$ is irrational.
\end{lemma}
We can use it to prove the following lemma. 
\begin{lemma}
\label{lem:angle_rot}
Consider a random vector $X \in \R^2$  that is symmetric under a rotation $R_\theta$ of angle $\theta \in [0, 2\pi)$
$$
X \stackrel{d}{=} R_\theta X. 
$$
Then $\frac{\theta}{2\pi} \in \mathbb{Q}$ or $X$ is invariant under all rotations. 
\end{lemma}
\begin{proof}
Assume $\frac{\theta}{2\pi} \not \in \mathbb{Q}$ and define the set $S = \left\{\frac{\theta k}{2\pi}  \bmod 1\left|  k \in \N \right. \right\}$.  $S$ is dense in $[0, 1)$ by Lemma \ref{lem:dense_in_unit_int}. Pick any point $s \in [0, 1)$ and choose a sequence $s_k \in S$ such that $\lim\limits_{k\rightarrow \infty}s_k = s$.  Consider any $f: \R^2 \rightarrow \R$ that is bounded and continuous. By repeatedly applying invariance, we have $\mathbb{E}\left(f(R_{2\pi s_k} X)\right) = \mathbb{E}\left(f(X)\right)$ for any $k$. Therefore,
$$
\mathbb{E}\left(f(R_{2\pi s} X)\right)  = \mathbb{E}\left(\lim\limits_{k \rightarrow \infty}f(R_{2\pi s_k} X)\right)  = \lim\limits_{k \rightarrow \infty}\mathbb{E}\left(f(R_{2\pi s_k} X)\right) =  \lim\limits_{k \rightarrow \infty}\mathbb{E}\left(f(X)\right)  = \mathbb{E}\left(f(X)\right). 
$$
where the second equality is the dominated convergence theorem. 
\end{proof}
We are now in a position to prove Theorem \ref{thm:main_result}. 
\begin{proof}[Proof of Theorem \ref{thm:main_result}]
Suppose $X$ with $\mathbb{E}\left\|X \right\|^2 _2< \infty$  is a solution of the fixed-point equation. We will proceed by showing that this implies that $U_u$ and $U_r$  are a.s.\ constant vectors, $U_r \in \operatorname{Im}(I - A_r)$ a.s.,  and $X$ satisfies (a)--(c). From $X \stackrel{d}{=} AX+U$ and subspace-invariance we can conclude the following equations:
$$
X_\bullet \stackrel{d}{=}A_\bullet X_\bullet + U_\bullet.
$$
We proceed by treating each block separately.

\medskip
\noindent {(a) Stable block $V_s$.} 
We can choose a norm with $\|A_s\|<1$ by Lemma \ref{lem:B_op_norm_bound} since $\rho(A_s) < 1$.  Define $\mathcal T(\mu):=(A_s)_\#\mu * \Law(U_s)$ on the metric space $(\mathcal P_2(V_s), W_2)$ with $*$ the convolution of distributions. Pushforward by $A_s$ is Lipschitz in the Wasserstein-2 metric induced by $\|\cdot\|$ with constant $\|A_s\|<1$ and convolution is 1-Lipschitz, so $\mathcal T$ is a strict contraction; by Banach's fixed-point theorem, there is a unique fixed point $\mu_s$. Now, let $m:=\mathbb{E}U_s$ and write $U^{(k)}_s=\tilde U^{(k)}_s+m$ with i.i.d.\ copies $\{U^{(k)}_s\}_{k\geq -1}$ of $U_s$. Set
\[
X_s^{\mathrm{det}}:=\sum_{k=0}^\infty A_s^k m,
\qquad
X_s^{\mathrm{rnd}}:=\sum_{k=0}^\infty A_s^k \tilde U_s^{(k)}.
\]
Since $\|A_s\|<1$, the Neumann series $\sum\limits_{k\ge0}A_s^k$ converges in operator norm and  $X_s^{\mathrm{det}}$ is well-defined. For the random series, 
\[
\mathbb{E}\big\|\sum_{k=N}^M A_s^k \tilde U_s^{(k)}\big\|^2
= \sum_{k=N}^M \mathbb{E}\|A_s^k \tilde U_s\|^2
\le \sum_{k=N}^M \|A_s^k\|^2 \mathbb{E}\|\tilde U_s^{(0)}\|^2,
\]
where the cross terms vanish because the summands are independent and centered. By the geometric series $\sum\limits_{k\ge0}\|A_s^k\|^2<\infty$, this shows Cauchy in $L^2$, hence $X_s^{\mathrm{rnd}}$ converges in $L^2$ by completeness. Defining
\[
X_s \ :=\ X_s^{\mathrm{det}} + X_s^{\mathrm{rnd}},
\]
by the continuous mapping theorem 
\[
A_s X_s + U_s^{(-1)}
{=}\ \sum_{k\ge1} A_s^k U_s^{(k-1)} + U_s^{(-1)}
\ \stackrel{d}{=} \ \sum_{k\ge0} A_s^k U_s^{(k)}
\ =\ X_s. 
\]
Thus $\Law(X_s)$ is the unique fixed point on $V_s$. 

\medskip
\noindent (b) Rotational block $V_r$. Choose a complex basis $v_1, \ldots, v_{d_r}$ of the complexified space $(V_r)^\C$  with $d_r$ its dimension and put $A_r$ into its complex Jordan form $\mathrm{diag}(J_1, \ldots, J_ {n_r})\in \C^{d_r \times d_r}$ with Jordan blocks $J_i$. 
For every Jordan block, the distributional equation 
$$
X_r^i \stackrel{d}{=}J_iX_r^i + U_r^i
$$
holds where $X_r^i, U_r^i$ are the coordinates of $X_r, U_r$ in the Jordan block $J_i$. Computing complex covariances yields 
$$
P = J _iP J_i^* + Q
$$
for $P$ and $Q$ the complex covariance matrices of $X_r^i$ and $U_r^i$. Apply the first part of Lemma \ref{lem:stein-noncontractive} to see from this that $Q = 0$ and that $P_{11}$ is the only nonzero index of $P$. Note that $P_{11}$ corresponds to the eigenvector in the Jordan chain of $J_i$. Applying this argument to every block shows that $U_r$ is a.s.\ constant and the only potentially non-a.s.-constant part of $X_r$ is the eigenvector component $X_r^{(1)}$. 
Note also by taking expectations that
$$
(I - A_r) \mathbb{E}(X_r) = \mathbb{E}(U_r) 
$$
which means that since $U_r$ is a.s.\ constant it must be a.s.\ in the image of $(I-A_r)$.   

\medskip
\noindent (c) Unstable block $V_u$ ($\rho(A_u)>1$). Using the same Jordan reduction argument as in (b), we arrive at the equation
$$
P = J P J^* + Q.
$$
Apply the second part of Lemma \ref{lem:stein-noncontractive} to conclude that $U_u$ and $X_u$ are a.s.\ constant. The distributional equation becomes an a.s.\ equation and we have  that a.s.
$$
X_u= (I - A_u)^{-1}U_u.
$$
Now, conversely, say that $U_u$ and $U_r$  are a.s.\ constant vectors with $U_r \in \operatorname{Im}(I - A_r)$ a.s. Construct the solution blockwise and make the blocks statistically independent so that blockwise satisfaction of the distributional equation is sufficient. In the stable and unstable blocks choose the solution as described in the theorem statement. Finally, choose $X_r$ constant such that it solves the linear equation
$$
(I - A_r)X_r = U_r
$$
a.s. It is clear that this is a valid solution from our previous argument, completing the proof. 
\end{proof}
Using Theorem \ref{thm:main_result}, we can prove Theorem \ref{thm:enkf_update_unique}. 
\begin{proof}[Proof of Theorem \ref{thm:enkf_update_unique}]
$\pi \in\mathcal{E}^\mathrm{EnKU}$ means that there is a distribution $\nu \in\Pp_2\left(\R^n\right)$  such that for $Z\sim \nu $ independent of $Y\sim \pi_Y$, $(X,Y) = (Z+ KY,Y) \sim \pi$ for $K= \Sigma_{XY}\Sigma_{YY}^\dagger$. 
Fix $y_\star\in \R^m$ and an affine map $\ell(x,y) = Ax +  By  + c$ that conditions exactly at $y_\star$.  This means that 
$$
AX + BY +c \stackrel{d}{=}Z + Ky_\star
$$
which can be rewritten as 
$$
A\overline Z+  (AK + B)\overline Y  + ((A-I)\E(Z) + (AK + B)\E(Y)  +c-Ky_\star) \stackrel{d}{=} \overline  Z.
$$
Defining $U = (AK + B)\overline Y  + ((A-I)\E(Z) + (AK + B)\E(Y)  +c-Ky_\star) $, this is equivalent to the following fixed point equation with $Z\indep U$:
$$
A\overline  Z + U \stackrel{d}{=}\overline Z.
$$

\noindent (a) $\pi \not \in \mathcal{S}_{\mathrm{cov}}$. Assume $\nu$ has a non-singular covariance meaning it does not have a constant linear component. By Theorem \ref{thm:main_result} (in the notation of the theorem), $Z_r^{(2)}$ and  $Z_u$ are a.s.\ constant. However, as we assumed that $Z$ has non-singular covariance, this means that the sum of generalized eigenspaces of $A$ with $|\lambda|>1$  is empty and  that $A$ is diagonalizable over the generalized eigenspace of all eigenvalues with magnitude $1$. In particular, $\rho(A) \leq 1$. 

\medskip
\noindent {(b) $\pi \not \in \mathcal{S}_{\mathrm{dec}}$.} Let $\pi \not \in \mathcal{S}_{\mathrm{dec}}$ and assume that $V_s$ is non-trivial, meaning that there is at least one complex eigenvalue $\lambda$ of $A$ with magnitude $|\lambda|<1$. There exists a left nonzero eigenvector $p \in \C^n$  such that
$$
p^\top A = \lambda p^\top.
$$
Plugging into the fixed point equation yields the 1D fixed point equation for $p^\top\overline{Z} $
$$
p^\top\overline{Z} \stackrel{d}{=} \lambda p^\top \overline{Z} + p^\top U.
$$
By point (a) of Theorem \ref{thm:main_result}  this implies that 
$$
p^\top \overline{Z} \stackrel{d}{=} \sum\limits_{k  = 0}^\infty \lambda^k p^\top U^{(k)}
$$
for i.i.d.\ copies $U^{(k)}$ of $U$. Writing $q^\top = p^\top(AK + B)$ and centering,  we can rewrite this as
$$
p^\top \overline  Z \stackrel{d}{=}  \sum\limits_{k = 0}^\infty\lambda^k q^\top \overline  Y^{(k)}. 
$$
for i.i.d.\ copies $Y^{(k)}$ of $Y$. However, this means that $\pi \in \mathcal{S}_{\mathrm{dec}}$ and so $V_s$ must have been trivial. This implies that $U$ is constant by Theorem \ref{thm:main_result} and therefore $(AK + B)P_Y = 0$ where $P_Y$ is the orthogonal projector onto the column space of $\text{Cov}(Y)$. This part of the statement is finalized by recognizing that $KP_Y = \Sigma_{XY}\Sigma_{YY}^\dagger P_Y$ as shown in the proof of Proposition \ref{prop:exact_set_kalman}.

\medskip
\noindent {(c) $\pi \not \in \mathcal{S}_{\mathrm{cyc}}$.} Finally, assume that $\pi \not \in \mathcal{S}_{\mathrm{cyc}}$. By projection, we have that
$$
\overline{Z}_r \stackrel{d}{=} A_r \overline{Z}_r +{U}_r.
$$
By Theorem \ref{thm:main_result}, $U_r$ is a.s.\ constant. Taking expectations shows that since $\overline{Z}_r$ is mean-zero, $U_r$ is a.s.\ $0$ so that we have 
$$
\overline{Z}_r \stackrel{d}{=} A_r \overline{Z}_r.
$$
Assume that $A_r$ has an eigenvalue $|\lambda| = 1$ with $\lambda \neq 1$ and derive a contradiction.  Consider the case $\lambda = -1$. Then there is a nonzero real $p$ such that $p^\top A = -p^\top$. This implies that 
$$
p^\top\overline{Z}_r \stackrel{d}{=} -  p^\top\overline{Z}_r 
$$
contradicting $\pi \not \in \mathcal{S}_{\mathrm{cyc}}$ for $Z_{\mathrm{cyc}} =(p^\top   P_r  Z, p^\top P_r  Z)$ and angle $\theta = \pi$. So $\lambda$ cannot be real. Write $\lambda = e^{i\theta}$ and let $p = p_1 + ip_2$ be a nonzero left eigenvector for $\lambda$. Note that neither $p_1$ nor $p_2$ can be zero as otherwise the equation $A_r^\top p_i = e^{i\theta} p_i$ would hold for one of $i = 1, 2$. This is impossible because the left-hand side is purely real while the right-hand side is not. Taking real and imaginary parts of $A_r^\top p =e^{i\theta}p$ yields
$$
A_r^\top p_1 = \cos \theta p_1 - \sin \theta p_2,\quad A_r^\top p_2 = \sin \theta p_1 + \cos \theta p_2.
$$
This implies that 
\begin{align*}
(p_1^\top A_r \overline{Z}_r, p_2^\top A_r \overline{Z}_r)^\top &\stackrel{d}{=}   ((\cos \theta p_1 - \sin \theta p_2)^\top\overline{Z}_r,(\sin \theta p_1 + \cos \theta p_2)^\top \overline{Z}_r)^\top \\
&\stackrel{d}{=}    R_\theta(p_1^\top \overline{Z}_r, p_2^\top \overline{Z}_r)^\top.
\end{align*}
By Lemma \ref{lem:angle_rot}, we have \(\theta = 2\pi s/t\) for some \(1 \leq s < t\), \(t \in \mathbb N\), with \(\gcd(s,t)=1\). Since the cyclic subgroup of \(\mathbb R/\mathbb Z\) generated by \(s/t\) is the same as the one generated by \(1/t\), iterating the preceding invariance relation shows that the same relation also holds with angle \(2\pi/t\). This implies \(\pi \in \mathcal S_{\mathrm{cyc}}\), contradicting \(\pi \notin \mathcal S_{\mathrm{cyc}}\). Hence we must have \(\lambda=1\). 
\end{proof}
Corollary \ref{cor:enkf_char} follows from Theorem \ref{thm:enkf_update_unique}. 
\begin{proof}[Proof of Corollary \ref{cor:enkf_char}]
Since $\pi \not \in \mathcal{S}_{\mathrm{cov}}$, Theorem \ref{thm:enkf_update_unique} implies that $\rho(A) \leq 1$  and $A$ is diagonal in the generalized eigenspace of all eigenvalues with magnitude $1$. Further, since $\pi \not \in \mathcal{S}_{\mathrm{dec}}$, the spectrum of  $A$ has no eigenvalues with magnitude smaller than $1$, so we can write $A = PDP^{-1}$ for $P, D\in\C^{n\times n}$ with $D$ diagonal and all diagonal entries of complex magnitude one. As $\pi \not \in \mathcal{S}_{\mathrm{cyc}}$,  $A$ has no eigenvalues with $|\lambda| = 1$ and  $\lambda \neq 1$, so $A = I$. Additionally, by $\pi \not \in \mathcal{S}_{\mathrm{dec}}$ and full covariance rank of $\Sigma_{YY}$,
$$B =-K $$ where $K = \Sigma_{XY}\Sigma_{YY}^\dagger$. Finally, we derive the value of $c$. Pick $\nu$ such that $(Z+ KY,Y) \sim \pi$ for $Z \sim \nu$ independent of $Y\sim \pi_Y$. 
Then, by exactness
$$
Z+KY +BY   +c\stackrel{d}{=} Z + Ky_\star.
$$
Taking expectations on both sides shows 
$$
c = Ky_\star
$$
and completes the proof. 
\end{proof}
Finally, we can prove Theorem \ref{thm:F_char}.
\begin{proof}[Proof of Theorem \ref{thm:F_char}]
Say $\pi \in \mathcal{F} \cap \mathcal{S}_\mathrm{nl-dec}^{\mathrm{c}}$. We show $\pi \in \mathcal{E}^{\mathrm{EnKU}}$. By Proposition \ref{prop:char_weak_set} there are measurable $d:\R^m\rightarrow \R^n$ and $ \nu\in\Pp_2\left(\R^n\right)$ such that 
$$
\pi_{X|Y=y_\star} = T_{d(y_\star)}\nu\text{ for all }y_\star \in \R^m.
$$
Letting $(X,Y) = (Z+d(Y),Y)$ for $Z\sim \nu$ independent of $Y \sim \pi_Y$, this means that $(X,Y) \sim \pi$.  Since $\pi$ has an exact weakly observation-dependent affine conditioning map there are  $A \in \R^{n\times n}$, $B\in\R^{n\times m}$, and $c:\R^m \rightarrow \R^n$  such that  
$$
AZ + Ad(Y)  + BY + c(y_\star) \stackrel{d}{=}Z + d(y_\star)
$$
$\pi_Y$-a.s. For any such $y_\star$ we can rewrite this as 
$$
A\overline Z + U \stackrel{d}{=}\overline Z
$$
for $U =(A-I) \mathbb{E}(Z) + Ad(Y)-d(y_\star) + BY + c(y_\star)$. Theorem \ref{thm:main_result} implies that $U_u$ and $U_r$  are a.s.\ constant vectors. Further, writing $A_s = P_sAP_s$, we have
\[
\overline Z_s \stackrel{d}{=} \sum_{k=0}^{\infty} A_s^{k}U^{(k)}_s
\]
for $U^{(k)}_s$ independent copies of $U_s$ that are chosen through independent copies $Y^{(k)}_s$ of $Y$. Say $V_s$ is nontrivial, then there is a nonzero eigenvector $p$ of $A_s^\top$ with eigenvalue $|\lambda| < 1$. This implies that 
$$
p^\top\overline Z_s \stackrel{d}{=} \sum_{k=0}^{\infty} \lambda^{k}p^\top U^{(k)}_s.
$$
for some $|\lambda|< 1$. We can expand 
$$
p^\top U^{(k)}_s  =p^\top P_sBY^{(k)}  +   p^\top P_s Ad(Y^{(k)})+ p^\top P_s\left((A-I) \mathbb{E}(Z)-d(y_\star)  + c(y_\star)\right)
$$
and defining  $b =  \frac{1}{1-\lambda}  {p^\top P_s\left((A-I) \mathbb{E}(Z)-d(y_\star)  + c(y_\star)\right)}$, $w^\top  =p^\top P_sB$, $u^\top = p^\top P_s A$, $v^\top = p^\top P_s$ we can rewrite this as 
$$
v^\top \overline Z  \stackrel{d}{=} \sum\limits_{k = 0}^\infty \lambda^k \left(w^\top Y^{(k)} + u^\top d(Y^{(k)})\right) + b.
$$
Since $v \neq 0$, this contradicts $\pi \in \mathcal{S}_\mathrm{nl-dec}^{\mathrm{c}}$ and thus we must have that $V_s$ is trivial. Therefore, $U$ (and in particular $Ad(Y) + BY$ ) is a.s.\ constant.  Further, $A\in \mathrm{GL}(n)$ as the stable block is trivial.  Therefore, almost surely
$$
d(Y) =- A^{-1}BY  +f
$$
for a constant $f \in \R^n$, meaning that $d$ is a.s.\ affine. However, then $\pi \in \mathcal{E}^{\mathrm{EnKU}}$. 
\end{proof}

\section*{Acknowledgments}
The authors acknowledge the use of AI tools for proofreading and language checking. All mathematical content and conclusions are solely those of the authors. 
\bibliographystyle{siamplain}
\bibliography{references}

\end{document}

%% file: ex_shared.tex
\usepackage{lipsum}
\usepackage{amsfonts}
\usepackage{graphicx}
\usepackage{epstopdf}
\usepackage{algorithmic}
\usepackage{booktabs}
\usepackage{colortbl} 
\usepackage{multirow} 
\definecolor{cmweakrow}{RGB}{236,243,252}    
\definecolor{cmweakedge}{RGB}{12,68,124}     
\definecolor{cmstrongrow}{RGB}{253,241,231}  
\definecolor{cmstrongedge}{RGB}{133,79,11}   
\usepackage{graphicx}
\usepackage{pgf}
\usepackage{float}
\usepackage{csquotes}
\usepackage[textsize=tiny]{todonotes}
\MakeOuterQuote{"}

\usepackage{amsmath}

\usepackage{amssymb}
\usepackage{tikz}
\usetikzlibrary{calc, matrix,positioning} 
\usepackage{hyperref}
\usepackage{cite}
\usepackage{adjustbox} 
\usepackage{import}
\usepackage{xifthen}
\usepackage{pdfpages}
\usepackage{transparent}

\newenvironment{psmallmatrix}
  {\left(\begin{smallmatrix}}
  {\end{smallmatrix}\right)}
\ifpdf
  \DeclareGraphicsExtensions{.eps,.pdf,.png,.jpg}
\else
  \DeclareGraphicsExtensions{.eps}
\fi
\newtheorem{example}{Example}[section]

\newcommand{\E}{\mathbb{E}}   

\newcommand{\Cov}{\mathrm{Cov}}
\newcommand{\Law}{\mathrm{Law}}

\newcommand{\Pp}{\mathcal{P}} 
\newcommand{\R}{\mathbb{R}}   
\newcommand{\N}{\mathbb{N}}   
\newcommand{\C}{\mathbb{C}}

\newcommand{\indep}{\perp \!\!\! \perp}

\makeatletter
\@ifundefined{ifarxiv}
  {\newcommand{\fjedit}[1]{\textcolor{blue}{#1}}}
  {\newcommand{\fjedit}[1]{\ifarxiv#1\else\textcolor{blue}{#1}\fi}}
\makeatother

\DeclareMathOperator{\spann}{span}
\usepackage{enumitem}
\setlist[enumerate]{leftmargin=.5in}
\setlist[itemize]{leftmargin=.5in}

\newsiamremark{remark}{Remark}
\newsiamremark{hypothesis}{Hypothesis}
\crefname{hypothesis}{Hypothesis}{Hypotheses}
\newsiamthm{claim}{Claim}

\headers{Unique characterization of the ensemble Kalman update}
{Frederic J.~N.~Jorgensen, Youssef Marzouk}

\author{Frederic J.\ N.\ Jorgensen, Youssef Marzouk}
\title{Exact affine conditioning beyond Gaussians: a unique characterization of the ensemble Kalman update\thanks{Submitted to the editors  April 26, 2026.
\funding{The work of the first author was supported by the Singapore-MIT Alliance for Research and Technology (SMART): Wafer-scale Integrated Sensing Devices based on Optoelectronic Metasurfaces (WISDOM) Interdisciplinary Research Group. The work of both authors was supported by the Office of Naval Research, SIMDA (Sea Ice Modeling and Data Assimilation) MURI, award number N00014-20-1-2595.}
}
}

\author{Frederic J.\ N.\ Jorgensen\thanks{Department of Mathematics, Massachusetts Institute of Technology
  (\email{fjorgen@mit.edu}).}
\and Youssef M.\ Marzouk\thanks{Laboratory for Information and Decision Systems, Massachusetts Institute of Technology
  (\email{ymarz@mit.edu}).
  }
}
\usepackage{amsopn}

\makeatletter
\newcommand*{\addFileDependency}[1]{
  \typeout{(#1)}
  \@addtofilelist{#1}
  \IfFileExists{#1}{}{\typeout{No file #1.}}
}
\makeatother

\newcommand*{\myexternaldocument}[1]{%
    \externaldocument{#1}%
    \addFileDependency{#1.tex}%
    \addFileDependency{#1.aux}%
}


%% file: connection_square_root.tex
\providecommand{\versioncite}[2]{\cite{#2}}
Here, we clarify how the affine transports used in the numerical experiments relate to square root filters. Ensemble square root filters (ESRFs) are deterministic variants of the ensemble Kalman filter that update the ensemble without requiring perturbed observations, typically improving stability and accuracy \versioncite{tippett2003ensemble, whitaker2002ensemble, bishop2001adaptive, hunt2007efficient}{SUPtippett2003ensemble, SUPwhitaker2002ensemble, SUPbishop2001adaptive, SUPhunt2007efficient}.  They are usually derived in settings where we have access to i.i.d.\ samples $\{X_i\}_{i = 1}^N$ (``forecast'') and we have the dependency $Y = HX + \xi$ with linear $H$,  independent mean-zero $\xi$, and  $\text{Cov}(\xi) = \Gamma$ finite. Define the forecast matrix $\hat X_f := \left(X_1 \ldots X_N \right)\in \R^{n \times N}$ with $n$ the state dimension and the forecast covariance $\hat C_f := \frac{1}{N-1}\hat X_f \left(I_N - \frac{1}{N}\mathbf{1}\mathbf{1}^\top\right)\hat X_f^\top$ where $\mathbf{1}$ is the vector with all entries $1$. The main idea in ESRFs  is to find an affine map 
$$
s: \R^{n\times N} \rightarrow\R^{n\times N}
$$
such that with $\hat X_a := s(\hat X_f)$ we have the following Gaussian-consistent moment conditions: 
\begin{align*}
\hat m_a& =    \hat m_f  + K\left(y_\star - H   \hat m_f \right) \\
\hat C_a&= C_a
\end{align*}
where
\begin{gather*}
\hat m_f := \frac{1}{N}\hat X_f \mathbf{1}, \quad     \hat m_a = \frac{1}{N}\hat X_a \mathbf{1}, \\
\hat C_a := \frac{1}{N-1}\hat X_a \left(I_N - \frac{1}{N}\mathbf{1}\mathbf{1}^\top\right)\hat X_a^\top,\quad  C_a = \hat C_f - \hat C_fH^\top (H\hat C_fH^\top + \Gamma)^{-1}H\hat C_f.
\end{gather*}
The prediction for the posterior $\pi_{X|Y= y_\star}$ in an ESRF is then  
$$
\widehat{\pi}^N_{X|Y=y_\star} = \frac{1}{N} \sum\limits_{i = 1}^N \delta_{\hat X_a ^i}
$$
where $\hat X_a^i$ are the columns of $\hat X_a$. There are multiple versions of ESRFs as the choice of such $s$ is not unique.
The most important versions of the ESRF are the ensemble transform Kalman filter (ETKF) \versioncite{bishop2001adaptive, tippett2003ensemble}{SUPbishop2001adaptive, SUPtippett2003ensemble} and the ensemble adjustment Kalman filter (EAKF):
\begin{enumerate}
    \item The ETKF is defined by requiring $s$ to operate on the anomaly matrix $\hat X_f^{(c)} = \hat X_f \left(I_N - \frac{1}{N}\mathbf{1}\mathbf{1}^\top\right)$ in ensemble space
$$
s(\hat X_f) = \hat X_f^{(c)}\hat T + \hat b\mathbf{1}^\top
$$
where $\hat b\in\R^n$ is a bias term uniquely determined by the first-order condition \versioncite{hunt2007efficient, bishop2001adaptive}{SUPhunt2007efficient, SUPbishop2001adaptive} . $\hat T\in \R^{N\times N}$ is therefore a matrix satisfying the second-moment condition 
$$
\hat T \hat T ^\top  =  I_N -\frac{1}{N-1} \left(\hat X_f^{(c)}\right)^\top H^\top\left(H\hat C_fH^\top + \Gamma\right)^{-1}H\hat X_f^{(c)}. 
$$
The unique principal square root of the right-hand side has been shown to be particularly stable. It is usually chosen for $\hat T$ \versioncite{wang2004better, hunt2007efficient,sakov2008implications, nerger2012unification}{SUPwang2004better, SUPhunt2007efficient,SUPsakov2008implications, SUPnerger2012unification}. This is unsurprising since it is the choice that is the ``least transformative,'' i.e.,
\[
\sqrt{M} = \arg\min\limits_{\tilde M \tilde M^\top = M}\left\|\tilde M - I\right\|_F,
\]
where $\sqrt{\cdot}$ is the principal square root, $M$ is positive semidefinite, and $\|\cdot \|_F$ is the Frobenius norm. Therefore, we let
$$
\hat T = \sqrt{I_N -\frac{1}{N-1} \left(\hat X_f^{(c)}\right)^\top H^\top\left(H\hat C_fH^\top + \Gamma\right)^{-1}H\hat X_f^{(c)}}
$$
be the principal square root. 
\item The EAKF, on the other hand, acts on the rows of the anomaly matrix \versioncite{anderson2001ensemble, tippett2003ensemble}{SUPanderson2001ensemble, SUPtippett2003ensemble}, meaning that 
$$
s(\hat X_f) = \hat A\hat X_f^{(c)}+ \hat b\mathbf{1}^\top.
$$
$\hat A\in \R^{n\times n}$ is therefore a matrix satisfying
$$
\hat A \hat C_f\hat A^\top = \hat C_a.
$$
The symmetric solution for this equation is given by 
$$
\hat A^{(1)} = \hat C_f^{\dagger/2}  \left(\sqrt{\hat C_f}   \hat C_a \sqrt{\hat C_f} \right)^{1/2}\hat C_f^{\dagger/2}
$$
with all square roots taken as the principal choice. Another possible choice is 
$$\hat{A}^{(2)} = \sqrt{\hat C_a }\hat C_f^{\dagger/2}.
$$
In practice, the following choice of square root is used more frequently instead \versioncite{anderson2001ensemble, tippett2003ensemble, grooms2020note}{SUPanderson2001ensemble, SUPtippett2003ensemble, SUPgrooms2020note}: let 
\[
\hat A^{(3)} = \hat X_f^{(c)}  C  (I + D)^{-1/2}  G^{\dagger} F^T,
\]
where $\hat X_f^{(c)} = F G U^T$ is the SVD  and $(\hat X_f^{(c)})^\top H^\top\Gamma^{-1} H\hat X_f^{(c)}= C D C^T$  is the eigenvalue decomposition  with the eigenvectors in the null space  arranged as the final columns of \(C\). 
\end{enumerate}
As we do not make the linear assumption \(Y=HX+\xi\) in our paper, we need to translate the expressions for \(\hat T\) and \(\hat A\) to this more general setting. Doing this for the EAKF is immediate. We simply replace the estimated analysis covariance with its population counterpart:
\[
C_a' = \hat C_f - \frac{1}{N-1}\hat X_f^{(c)} \bigl(\hat Y_f^{(c)}\bigr)^\top
\left( \hat Y_f^{(c)} \bigl(\hat Y_f^{(c)} \bigr)^\top \right)^\dagger
\hat Y_f^{(c)} \bigl(\hat X_f^{(c)}\bigr)^\top .
\]
This shows directly that  \(L_{y_\star}^{\text{OT}}\)  and  \(L_{y_\star}^{\text{D}}\)   implement the EAKF updates \(\hat A^{(1)}\) and \(\hat A^{(2)}\). For the ETKF, the idea is similar. Starting with $\hat T$, we note that the expression containing $H\hat X_f^{(c)}$ or $\Gamma$ involves prior knowledge of the covariance structure, namely the assumption $Y = HX + \xi$.  The generalization of $\hat T$ to nonlinear settings is therefore 
$$
\hat T' =  \sqrt{I_N - \left(\hat Y^{(c)}_f\right)^\top \left(   \hat Y^{(c)}_f\left(\hat Y^{(c)}_f\right)^\top \right)^{\dagger}  \hat Y^{(c)}_f}
$$
with $\hat Y^{(c)}_f \in \R^{m \times N}$ the centered ensemble matrix of the observations $Y_i$.  As the second term in $\hat T'$ is a projection onto the row space of $\hat Y^{(c)}_f$, $\hat T'$ is the principal square root of a projector onto the orthogonal complement of  the row space of $\hat Y^{(c)}_f$. Orthogonal projectors are their own principal square roots  and therefore
$$
\hat T' =  I_N - \left(\hat Y^{(c)}_f\right)^\top \left(   \hat Y^{(c)}_f\left(\hat Y^{(c)}_f\right)^\top \right)^{\dagger}  \hat Y^{(c)}_f.
$$
However, now note that 
\begin{align*}
    \hat X_f^{(c)}\hat T' & = \hat X_f^{(c)} - \hat X_f^{(c)} \left(\hat Y^{(c)}_f\right)^\top \left(   \hat Y^{(c)}_f\left(\hat Y^{(c)}_f\right)^\top \right)^{\dagger}  \hat Y^{(c)}_f\\
     &= \hat X_f^{(c)} - \hat \Sigma_{XY}\hat \Sigma_{YY}^\dagger  \hat Y^{(c)}_f.
\end{align*}
This shows that this generalized ETKF and the  EnKU perform the same update. In that sense, everything we say above about the EnKU applies to the ETKF as they are the same outside the linear-Gaussian setting.
\begin{remark}
The presented ``generalizations'' of the ESRF are not meant to be  good filtering methods.  In fact, they forfeit the main advantage of ESRFs, namely the deterministic (non-stochastic) update.  The point of introducing them above lies instead in providing insight into the bias inherent in ESRF methods and clarifying 
their connection to the EnKU. 
\end{remark}

%% file: finite_sample_effects.tex
The error of an affine update applied to a finite ensemble can be upper-bounded by a mean-field \emph{bias} and a finite-sample \emph{variance} term (see also Equation~\eqref{eq:variance_bias_decomposition}), and in practice a posterior estimate is accurate if both are small at the ensemble sizes employed. The experiments in Section~\ref{section:num_exp} investigate the \emph{bias}, i.e., the discrepancy between the mean-field posterior induced by an affine transport and the true posterior. Here, for the same three classes of data-generating models and the same observation \(y_\star=(0.4,-0.2)^\top\), we instead study the \emph{variance} term
\[
\mathbb{E}\,W_2\!\left(\widehat{\pi}^N_{X\mid Y=y_\star},\ \tilde{\pi}_{X\mid Y=y_\star}\right)
\]
which is the discrepancy between the predicted finite-ensemble analysis distribution $\widehat{\pi}^N_{X\mid Y=y_\star}$ and the
mean-field pushforward measure $\tilde{\pi}_{X\mid Y=y_\star}$ it approximates. This quantity vanishes as $N\to\infty$ for all
methods, so the comparison concerns finite-sample behavior.

Experiments~1--3 use exactly the same data-generating models, observation, and Monte Carlo sampling as in Section~\ref{section:num_exp} and for each replicate, each finite-ensemble output is compared with the corresponding \textit{method-specific} mean-field limit. For each method $M$, this limit is
\[
\widetilde{\pi}^{M}_{X\mid Y=y_\star}=(L^M_{\pi,y_\star})_\sharp\pi,
\qquad M \in \{\text{EnKU}, \text{D}, \text{OT}\}.
\]
It is computed by applying the corresponding population map (which is analytically available) to an independent ensemble of size \(N_{\mathrm{ref}}=49{,}152\).
\begin{figure}[htbp]
\centering
\begin{minipage}{0.49\textwidth}
\includegraphics[width=\linewidth]{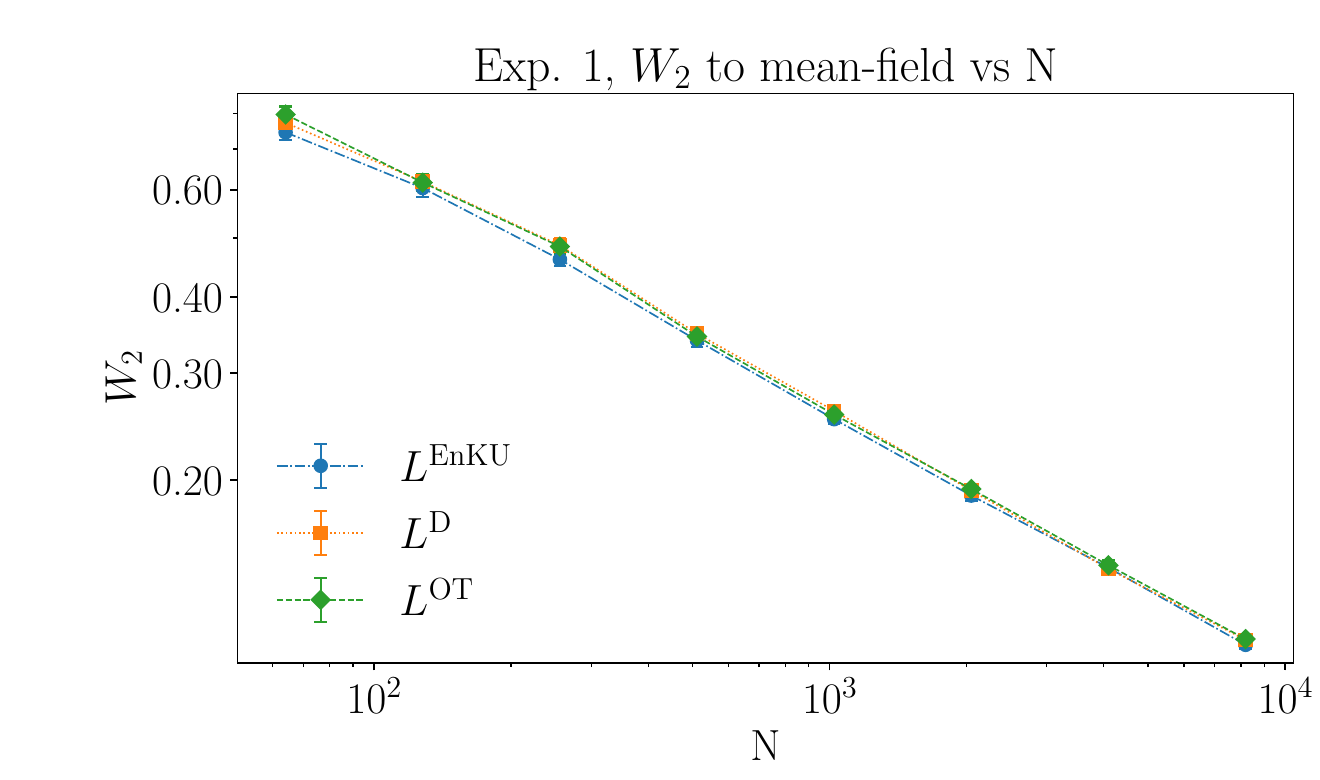}
\end{minipage}\hfill
\begin{minipage}{0.49\textwidth}
\includegraphics[width=\linewidth]{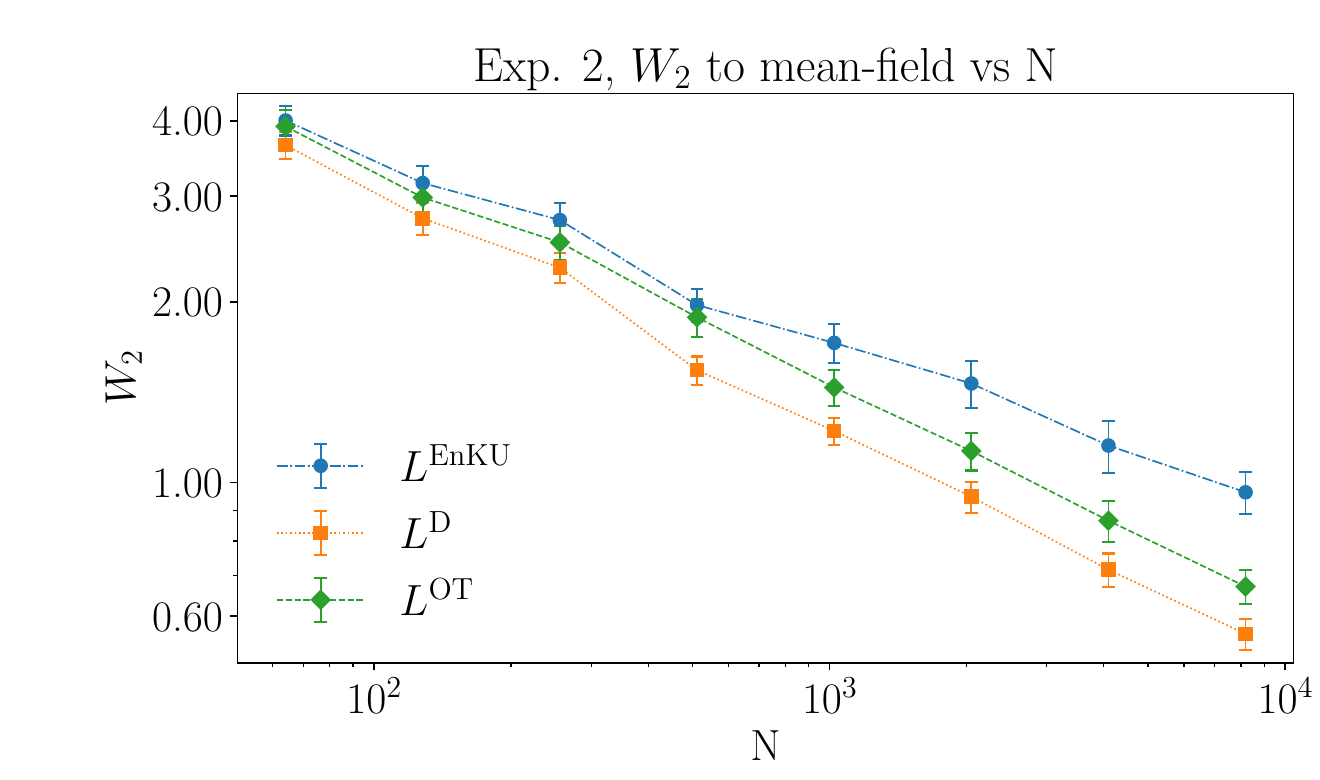}
\end{minipage}\\
\begin{minipage}{0.49\textwidth}
\includegraphics[width=\linewidth]{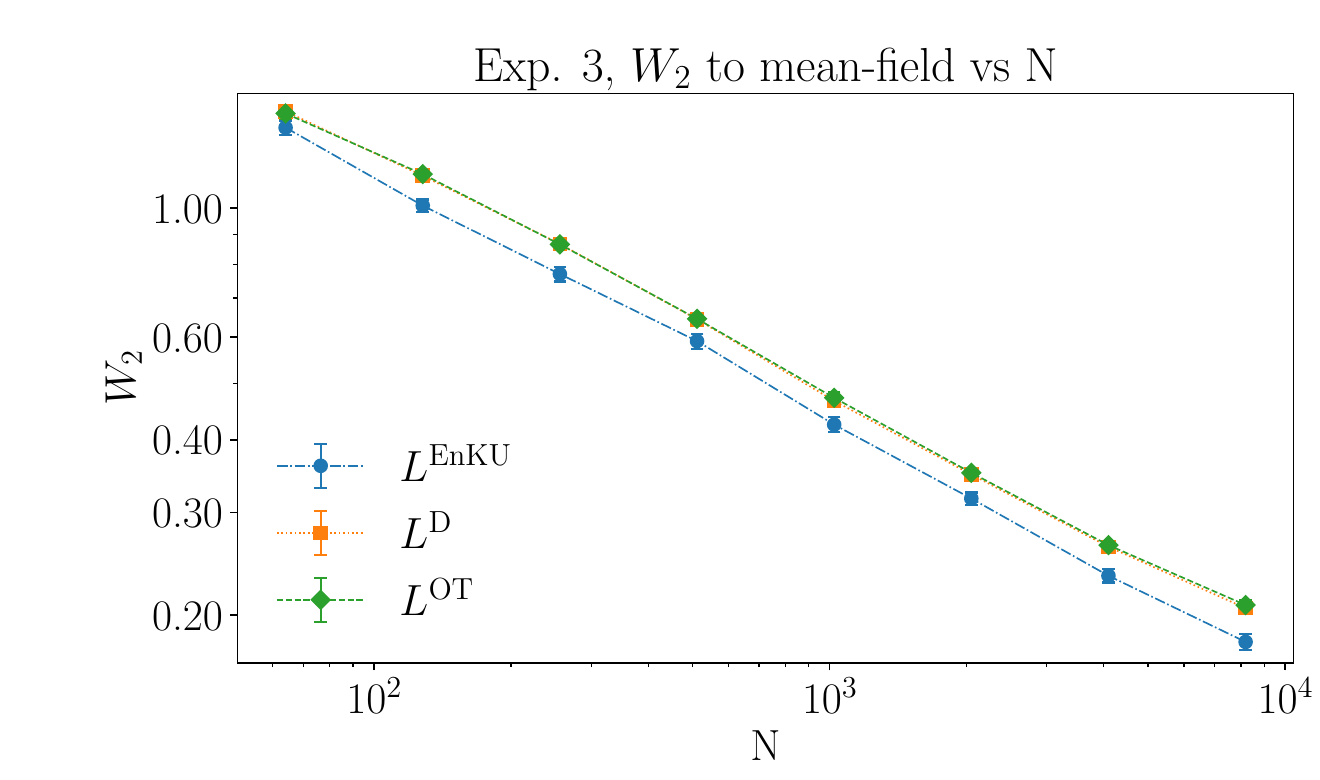}
\end{minipage} \hfill
\caption{\fjedit{Finite-ensemble error for the three affine updates.  Each panel plots the Monte Carlo mean \(\pm\) standard error of
\(W_2(\widehat{\pi}^{N,M}_{X\mid Y=y_\star},
\widetilde{\pi}^{M}_{X\mid Y=y_\star})\) as a function of ensemble size \(N\).}}
\label{fig:w2_to_meanfield}
\end{figure}
In the resulting Figure~\ref{fig:w2_to_meanfield} plotting $\mathbb{E}\,W_2\!\left(\widehat{\pi}^N_{X\mid Y=y_\star},\ \tilde{\pi}_{X\mid Y=y_\star}\right)$, all curves decrease with \(N\). Across all experiments, no single map is uniformly best; the ranking of finite-ensemble variance is example-dependent. Systematic differences between the stochastic and square-root choices may emerge in the small-\(N\), high-dimensional examples that motivate square-root filters in practice, but we do not study these here \cite{whitaker2002ensemble, tippett2003ensemble, lawson2004implications}.

%% file: notes_inverse_problems.tex
\subsection{\fjedit{Interpretation for additive independent-noise models}}
\label{subsec:additive_noise_models}
In inverse problems and filtering, the joint law of $(X,Y)$ is often specified through
an additive independent-noise model
\begin{equation}
\label{eq:additive_noise_model}
Y=h(X)+\varepsilon,
\qquad X\indep\varepsilon.
\end{equation}
The next result translates the exactness condition in the EnKU of  \Cref{prop:exact_set_kalman} into conditions on the observation map $h$, the prior density (or forecast marginal density) $p_X$, and the noise density
$p_\varepsilon$, assuming mild regularity. 

\begin{theorem}
\label{thm:additive_noise_differential_characterization}
Let $X\in\R^n$ and $\varepsilon\in\R^m$ be independent random vectors with strictly positive densities
\[
p_X\in C^2(\R^n),
\qquad
p_\varepsilon\in C^2(\R^m).
\]
Let $h\in C^2(\R^n;\R^m)$, assume that $X$, $\varepsilon$, and $h(X)$ have finite second moments, and set
\[
Y=h(X)+\varepsilon,
\qquad
K:=\Sigma_{XY}\Sigma_{YY}^{-1},
\]
where we assume that $\Sigma_{YY}$ is  positive definite. Define
\[
g_X:=\log p_X,
\qquad
g_\varepsilon:=\log p_\varepsilon,
\qquad
\varphi(x,y):=g_X(x)+g_\varepsilon\bigl(y-h(x)\bigr).
\]
Then the following statements are equivalent:
\begin{enumerate}
\item the EnKU is exact for the joint law of $(X,Y)$;
\item there exist positive functions $a\in C^2(\R^n)$ and
$b\in C^2(\R^m)$ such that
\begin{equation}
\label{eq:additive_noise_factorization}
p_X(x)p_\varepsilon\bigl(y-h(x)\bigr)
=a(x-Ky)b(y),
\qquad (x,y)\in\R^n\times\R^m;
\end{equation}
\item the following partial differential equation holds
\begin{equation}
\label{eq:additive_noise_compact_de}
D^2_{xy}\varphi(x,y)+D^2_{xx}\varphi(x,y)K=0,
\qquad (x,y)\in\R^n\times\R^m
\end{equation}
\end{enumerate}
where $[D^2_{xy}\varphi]_{i,j}=\partial_{x_i}\partial_{y_j}\varphi$ and $[D^2_{xx}\varphi]_{i,j}=\partial_{x_i}\partial_{x_j}\varphi$.
\end{theorem}

\begin{proof}
By \Cref{prop:exact_set_kalman} and the canonical choice used in its proof,
exactness is equivalent to $Z:=X-KY\indep Y.$ Under the shear $(x,y)\mapsto(z,y):=(x-Ky,y)$ with Jacobian determinant  equal to one, the density of $(Z,Y)$ is
\[
q(z,y)=p_X(z+Ky)p_\varepsilon\bigl(y-h(z+Ky)\bigr).
\]
Thus $Z\indep Y$ is equivalent to Equation~\eqref{eq:additive_noise_factorization} by continuity. The associated cross-product identity
\[
q(z,y)q(z',y')=q(z,y')q(z',y)
\]
then also holds everywhere. Fixing $(z',y')=(z_0,y_0)$ gives
positive $C^2$ factors
\[
a(z):=q(z,y_0),
\qquad
b(y):=\frac{q(z_0,y)}{q(z_0,y_0)}.
\]

Set $\psi(z,y):=\log q(z,y)=\varphi(z+Ky,y)$. On the connected domain
$\R^n\times\R^m$, the density $q$ factors if and only if
$D^2_{zy}\psi=0$, where $[D^2_{zy}\psi]_{i,j}=\partial_{z_i}\partial_{y_j}\psi$. The chain rule gives
\[
D^2_{zy}\psi(z,y)
=D^2_{xy}\varphi(z+Ky,y)+D^2_{xx}\varphi(z+Ky,y)K,
\]
which proves \eqref{eq:additive_noise_compact_de}. 
\end{proof}
The partial differential equation~\eqref{eq:additive_noise_compact_de} is highly non-generic, showing that general inverse problems do not lie in the exact set of the EnKU. This is particularly dramatic in one dimension, as the following proposition shows. 
\subsection{\fjedit{One-dimensional case}}
In one dimension, specifically $n=m=1$,
the preceding differential condition admits a nearly complete
classification.  Note that we may assume without loss of generality that the EnKU gain equals one: any scalar model \((h,p_X,p_\varepsilon)\) with nonzero gain \(K\) is mapped to a gain-one model with observation map \(Kh\) and noise \(K\varepsilon\) by the invertible reparametrization \(Y\mapsto KY\), which preserves exactness since \(Z=X-KY\) is unchanged.
\begin{proposition}
\label{prop:scalar_additive_noise_classification}
Let $X$ and $\varepsilon$ be independent real-valued random
variables with finite second moments, let $h\in C^3(\R)$, assume that $Y=h(X)+\varepsilon$ results in a finite gain taking the value $K = 1$, and suppose that $X$ and
$\varepsilon$ have strictly positive densities 
\[
p_X=e^{g_X},
\qquad
p_{\varepsilon}=e^{g_{\varepsilon}},
\qquad
g_X,g_{\varepsilon}\in C^3(\R).
\]
Further assume that the noise density is strictly log-concave
\begin{equation}
\label{eq:strict_log_concavity_noise}
g_{\varepsilon}''(u)<0,
\qquad u\in\R.
\end{equation}
Then the EnKU is exact for the joint law of $(X,Y)$ if and only if  one 
of the following two cases holds.
\begin{enumerate}
\item[\emph{(i)}] \emph{Affine--Gaussian case.}
There exist $a\in(0,1)$, $b,\mu_X,\mu_{\varepsilon}\in\R$, and
$\sigma_X>0$ such that
\[
h(x)=ax+b,
\qquad
X\sim\mathcal N(\mu_X,\sigma_X^2),
\qquad
\varepsilon
\sim\mathcal N\bigl(\mu_{\varepsilon},
 a(1-a)\sigma_X^2\bigr).
\]

\item[\emph{(ii)}] \emph{Logistic--Gamma case.} 
There exist
\[
c\in\R\setminus\{0\},
\qquad
q,\lambda,r,s>0,
\qquad
b\in\R,
\]
such that
\begin{equation}
\label{eq:generalized_softplus_map}
h(x)=x-\frac1c\log\bigl(1+qe^{cx}\bigr)+b,
\end{equation}
\begin{equation}
\label{eq:generalized_logistic_prior}
p_X(x)
=\frac{|c|q^s}{\mathrm B(s,r)}
\frac{e^{scx}}{(1+qe^{cx})^{r+s}},
\end{equation}
and
\begin{equation}
\label{eq:log_gamma_noise}
p_{\varepsilon}(u)
=\frac{|c|\lambda^{r+s}}{\Gamma(r+s)}
\exp\bigl((r+s)cu-\lambda e^{cu}\bigr)
\end{equation}
where $\mathrm B(s,r)$ is the Beta function and $\Gamma(r+s)$ the Gamma function.  Equivalently,
\[
qe^{cX}\sim\operatorname{BetaPrime}(s,r),
\qquad
\lambda e^{c\varepsilon}
\sim\operatorname{Gamma}(r+s,1),
\]
where $\operatorname{BetaPrime}(s,r)$ has density $t^{s-1}(1+t)^{-r-s}/\mathrm B(s,r)$ on $(0,\infty)$.
\end{enumerate}
\end{proposition}
\begin{proof}
For the gain-one model, \(\varphi(x,y)=g_X(x)+g_{\varepsilon}(y-h(x))\).
By \Cref{thm:additive_noise_differential_characterization}, exactness is equivalent to the scalar
form of \eqref{eq:additive_noise_compact_de}, namely
\(\partial_{xy}\varphi+\partial_{xx}\varphi=0\), where
\(u=y-h(x)\) and
\[
\partial_{xy}\varphi=-h'(x)g_{\varepsilon}''(u),
\qquad
\partial_{xx}\varphi=g_X''(x)-h''(x)g_{\varepsilon}'(u)
+h'(x)^2g_{\varepsilon}''(u).
\]
Thus exactness is equivalent to
\begin{equation}
\label{eq:gain_one_scalar_de}
g_X''(x)
=h''(x)g_{\varepsilon}'(u)
+h'(x)\bigl(1-h'(x)\bigr)g_{\varepsilon}''(u),
\qquad (x,u)\in\R^2.
\end{equation}
If $h(x)=ax+b$ for $a,b \in \R$, then
$g_X''(x)=a(1-a)g_{\varepsilon}''(u)$ (in particular, $\operatorname{Var}(\varepsilon)
=a(1-a)\operatorname{Var}(X)$). The cases
$a\in\{0,1\}$ contradict integrability of $e^{g_X}$. Hence both sides are
constant and $p_X$ and $p_\varepsilon$ are Gaussian densities. Further, $g_X''(x)$ and $g_\varepsilon''(x)$ must both be negative constants by integrability which gives $a\in(0,1)$. 

Assume now that $h$ is nonaffine. There exists $x_0$ (and a small interval around) with
$h'(x_0)(1-h'(x_0))\neq0$; otherwise continuity would force
$h'\equiv0$ or $h'\equiv1$. Differentiating \eqref{eq:gain_one_scalar_de} with respect to $u$ gives
\begin{equation}
\label{eq:scalar_separation_identity}
h''(x)g_{\varepsilon}''(u)
+h'(x)\bigl(1-h'(x)\bigr)g_{\varepsilon}'''(u)=0.
\end{equation}
Since
$g_{\varepsilon}''<0$, separation in
\eqref{eq:scalar_separation_identity} yields a constant $c\neq0$ (by non-affinity) such that
\begin{equation}
\label{eq:separated_odes}
g_{\varepsilon}'''=c g_{\varepsilon}'',
\qquad
h''+ch'(1-h')=0
\end{equation}
where the second equation follows globally from the first. 
The first equation gives $g_\varepsilon''(u)=-\lambda c^2e^{cu}$ for some
$\lambda>0$, where the sign follows from strict log-concavity. Integrating
twice gives $g_\varepsilon(u)=C+\beta u-\lambda e^{cu}$. Since
$e^{g_\varepsilon}$ is integrable, we must have that \(\alpha:=\beta/c>0\).
Moreover, with the substitution $t=\lambda e^{cu}$ and by the definition of the Gamma function
\[
\int_\R e^{\alpha cu-\lambda e^{cu}}\,du
=\frac{\Gamma(\alpha)}{|c|\lambda^\alpha},
\]
so normalizing gives
\begin{equation}
\label{eq:log_gamma_intermediate}
g_{\varepsilon}(u)
=\log\frac{|c|\lambda^\alpha}{\Gamma(\alpha)}
+\alpha cu-\lambda e^{cu},
\end{equation}
where \(\alpha>0\). 
Writing $v=h'$, the second equation in \eqref{eq:separated_odes} is the logistic differential equation
$v'=-cv(1-v)$. By the standard logistic equation solution formula
\cite[Sec.~2.5]{boyce2017elementary}, its nonconstant solutions (as required by non-affine $h$) defined on all of $\R$ are
\[
v(x)=\frac1{1+qe^{cx}},
\qquad q>0,
\]
and integration gives \eqref{eq:generalized_softplus_map}. Substitution into
\eqref{eq:gain_one_scalar_de} yields
\[
g_X''(x)=-\alpha c^2v(x)(1-v(x)),
\]
so that by integrating twice
\[
g_X(x)=scx-\alpha\log(1+qe^{cx})+C
\]
for some $s\in\R$. Integrability is equivalent to $0<s<\alpha$. 
With \(t=qe^{cx}\), by definition of the Beta function
\[
\int_\R e^{scx}(1+qe^{cx})^{-\alpha}\,dx
=\frac{1}{|c|q^s}\mathrm B(s,\alpha-s).
\]
The change of variables $r:=\alpha-s$ gives
\eqref{eq:generalized_logistic_prior}--\eqref{eq:log_gamma_noise}. 
Conversely, direct substitution verifies \eqref{eq:gain_one_scalar_de} in both
cases, and hence exactness by \Cref{thm:additive_noise_differential_characterization}. 
\end{proof}

\begin{remark}
\label{rem:scalar_original_coordinates}
In the zero-gain case $K=0$ with log-concave noise $(\log p_\varepsilon)''<0$, the scalar form of
\eqref{eq:additive_noise_compact_de} implies $h'\equiv0$ and the EnKU is exact trivially.
\end{remark}

%% file: bias_outside_exact_set.tex
\Cref{thm:additive_noise_differential_characterization} shows a generic additive-noise model $Y=h(X)+\varepsilon$ need not lie in the exact set $\mathcal{E}^{\mathrm{EnKU}}$  of the EnKU. A relevant question in practice is therefore how fast the EnKU error grows as the joint law leaves $\mathcal{E}^{\mathrm{EnKU}}$. 

For this purpose, we consider a two-dimensional inverse problem in which the joint law is slowly moved away from  $\mathcal{E}^{\mathrm{EnKU}}$.  For  $s \in [0,1]$, consider the mixture prior
\[
    X \sim \frac14
    \sum_{\varepsilon_1,\varepsilon_2\in\{-1,1\}}
    \mathcal N((\varepsilon_1 s,\varepsilon_2 2s)^\top,\Sigma_c),
\]
where
\[
    \Sigma_c =
    \begin{pmatrix}
        0.5^2 & 0.2(0.5)(0.3) \\
        0.2(0.5)(0.3) & 0.3^2
    \end{pmatrix}.
\]
The observation is
\[
    Y = X_1+\eta,\qquad \eta\sim \mathcal N(0,0.4^2).
\]
The linear-Gaussian $s=0$ case lies in $\mathcal{E}^{\mathrm{EnKU}}$, while increasing \(s\) moves the joint law progressively outside of $\mathcal{E}^{\mathrm{EnKU}}$. For each \(s\), the conditioning value \(y_\star(s)\) is chosen as the
\(0.9\)-quantile of the predictive distribution of \(Y\) to make the observation comparably informative across $s$. 

For each \(s\), we compare the pushforwards of joint samples under
\(L^{\mathrm{EnKU}}\), \(L^{\mathrm D}\), and \(L^{\mathrm{OT}}\) as in \Cref{subsection:within_enku} against
independent samples from this exact posterior. Since the likelihood is linear and Gaussian, the reference posterior \(\pi_{X\mid Y=y_\star(s)}\) is available as a four-component Gaussian mixture from which we can generate true posterior samples.  
The discrepancy plotted in 
\Cref{fig:inverse_problem_bias_outside_exact_set} is the max-sliced
Wasserstein-2 distance \(\overline{W}_2\), estimated using \(40\)
Monte Carlo replications with \(N=65,536\) predicted particles and \(3N\)
reference posterior particles.\footnote{For measures \(\mu,\nu\) on \(\R^2\),
\(\overline{W}_2(\mu,\nu):=\sup_{\|\theta\|=1}W_2(\mu_\theta,\nu_\theta)\),
where \(\mu_\theta=\Law(\theta^\top X)\) for \(X\sim\mu\).  We approximate
the supremum by \(64\) evenly spaced directions in \([0,\pi)\), avoiding
prohibitively expensive large-particle empirical \(W_2\).}  

\begin{figure}[!tbp]
\centering
\begin{minipage}{0.49\textwidth}
    \includegraphics[width=\linewidth,clip,trim=0 0 0 0]{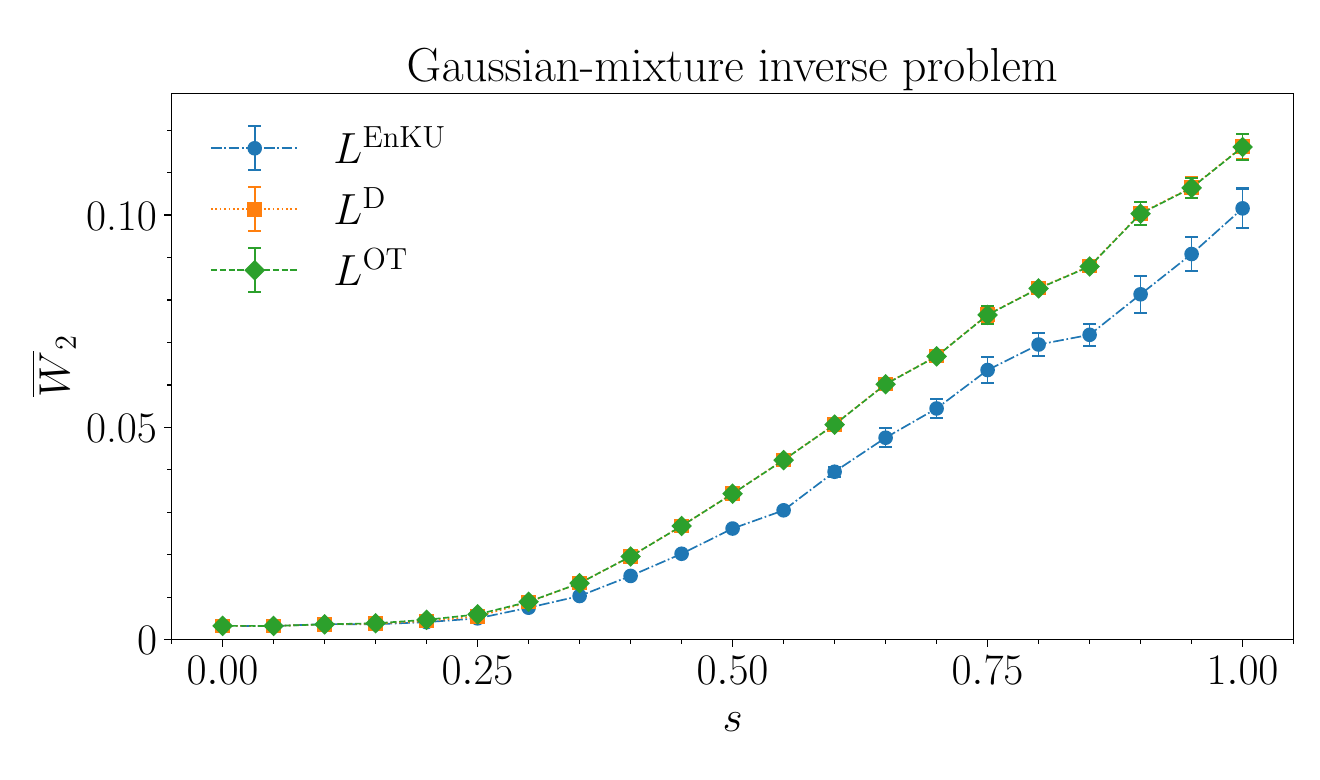}
\end{minipage}\hfill
\includegraphics[width=\linewidth,clip,trim=0 0 0 0]{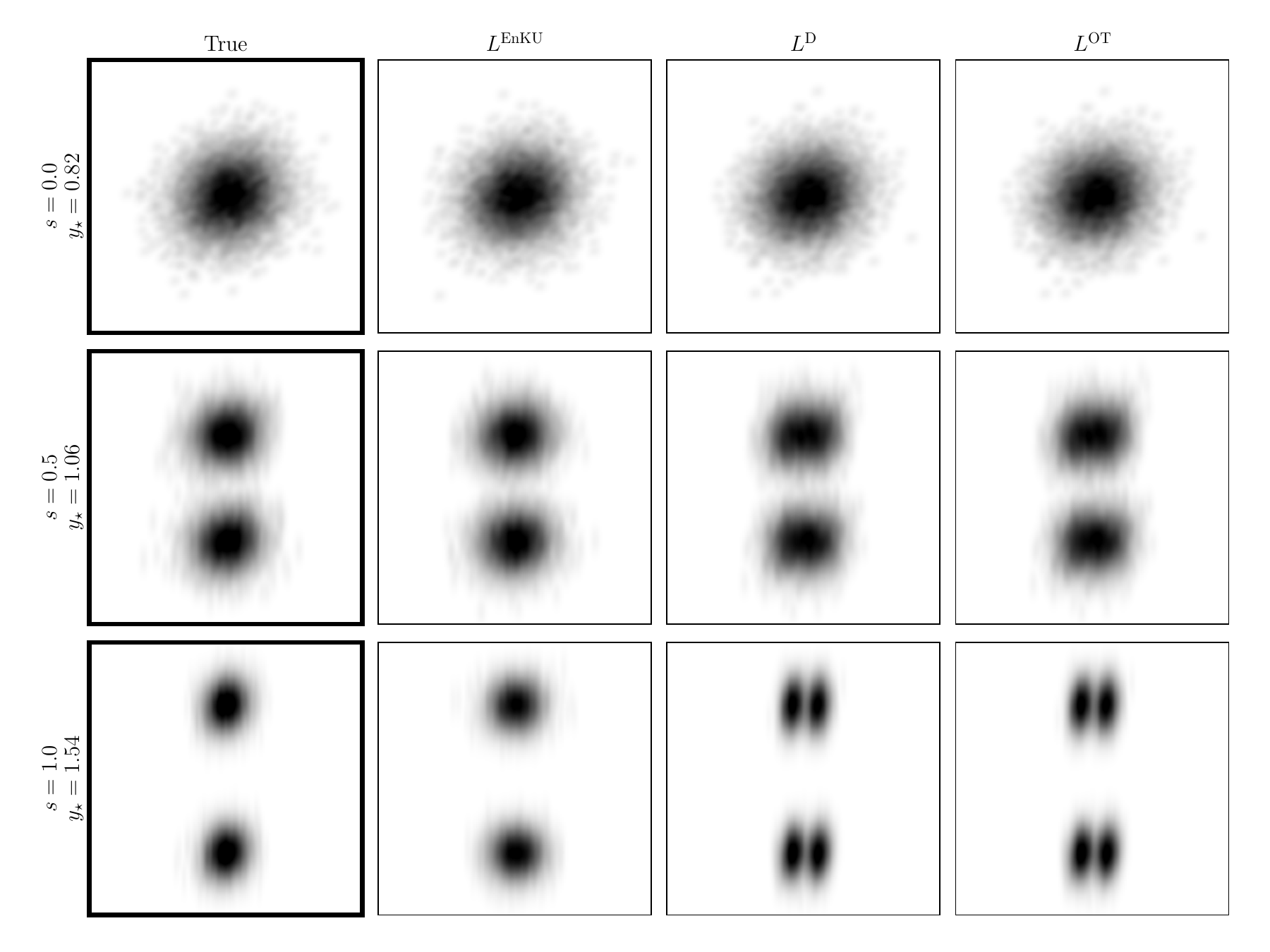}
\caption{\fjedit{
Bias outside the exact set in a Gaussian-mixture inverse problem.
Top: \(\overline{W}_2\) error between the exact posterior and the
posterior ensembles produced by \(L^{\mathrm{EnKU}}\), \(L^{\mathrm D}\), and
\(L^{\mathrm{OT}}\), as a function of the departure parameter \(s\).
As expected, at \(s=0\) all three methods agree with the exact posterior up to Monte Carlo
error.  As \(s\) increases and the prior becomes increasingly multimodal, the EnKU
exhibits smaller mean-field bias than \(L^{\mathrm D}\) and
\(L^{\mathrm{OT}}\).  Bottom: kernel density estimates of the true posterior and the three
affine-update pushforwards at \(s=0\), \(s=0.5\), and \(s=1\).  Densities are
visualized using a two-dimensional Gaussian kernel density estimator with
Scott's rule and bandwidth multiplier \(0.6\). 
}}
\label{fig:inverse_problem_bias_outside_exact_set}
\end{figure}
We see in \Cref{fig:inverse_problem_bias_outside_exact_set} that as the mixture separation grows, \(L^{\mathrm{EnKU}}\) remains closer to the exact posterior than the deterministic and OT square-root affine updates. This shows that, even when the joint law is not in \(\mathcal E^{\mathrm{EnKU}}\), the larger exactness set of the EnKU can translate into a smaller mean-field bias.

%% file: meagreness.tex
We prove the statement from Remark \ref{rem:small_symmetry}, restated here in the form of a proposition. 
\begin{proposition}
\label{prop:small_symmetry}
Define
\[
\mathcal E^{\mathrm{EnKU}}_0
:=
\bigl\{
\pi \in \mathcal E^{\mathrm{EnKU}} \, \big| \, \Sigma_{YY}(\pi)\text{ is invertible}\bigr\},
\]
and endow it with the relative \(W_2\) topology. Then $\mathcal S_{\mathrm{dec}} \cap \mathcal E^{\mathrm{EnKU}}_0$ and $\mathcal S_{\mathrm{cyc}} \cap \mathcal E^{\mathrm{EnKU}}_0$ are meagre subsets of \(\mathcal E^{\mathrm{EnKU}}_0\). 
\end{proposition}
\begin{proof}[Proof of Proposition \ref{prop:small_symmetry}]
By Proposition \ref{prop:exact_set_kalman}, every \(\pi \in \mathcal E^{\mathrm{EnKU}}_0\) admits a representation
\[
\pi=\Law(Z+MY,Y),
\qquad Z\indep Y,
\]
with \(M\in \R^{n\times m}\). Taking the cross-covariance of the left- and right-hand sides yields
\[
M \Sigma_{YY}(\pi) = \Sigma_{XY}(\pi),
\]
which uniquely determines the decomposition. Now define 
\[
\mathcal X_0
:=
\Pp_2(\R^n)\times \Pp_{2,\mathrm{nd}}(\R^m)\times \R^{n\times m},
\]
with the nondegenerate set
\[
\Pp_{2,\mathrm{nd}}(\R^m)
:=
\bigl\{
\nu\in \Pp_2(\R^m)\big| \Sigma(\nu)\text{ is invertible}
\bigr\}
\]
where $\Sigma(\nu)$ is the full covariance matrix of $\nu$,  and let
\[
\Phi(\mu,\nu,M):=\Law(Z+MY,Y),
\qquad Z\sim \mu, Y\sim \nu, Z\indep Y.
\]
$\Phi$ is continuous by a standard $W_2$-stability triangle inequality argument. 
Its inverse is
\[
\Phi^{-1}(\pi)
=
\Bigl(
(x-My)_\sharp \pi,
\pi_Y,
M
\Bigr),
\qquad
M=\Sigma_{XY}(\pi)\Sigma_{YY}(\pi)^{-1}.
\]
The inverse map is also continuous. For the third component, continuity follows since second moments depend continuously on $\pi$ in $W_2$, and matrix inversion is continuous on the open set of invertible matrices. Continuity of the second component is immediate. For the first component, continuity follows from a standard $W_2$ stability argument: linear pushforwards are continuous in $W_2$, and combining this with a triangle inequality to also control the $M$-term yields the claim. Hence
\[
\Phi:\mathcal X_0 \rightarrow \mathcal E^{\mathrm{EnKU}}_0
\]
is a homeomorphism. It therefore suffices to prove meagreness of the pullbacks of the sets under consideration to $\mathcal X_0$. 

Further, let \(\mathcal A\subset \Pp_2(\R^n)\) be the set of all finitely supported distributions
\[
\mu=\sum_{i=1}^N p_i \delta_{x_i}
\]
with $p_i>0$ and $x_i$ distinct such that:
\begin{enumerate}
\item the centered support spans \(\R^n\), i.e.,
\[
\spann\{x_i-m(\mu):1\le i\le N\}=\R^n,
\qquad
m(\mu):=\int xd\mu(x),
\]
\item all subset sums of the weights are distinct:
\[
\sum_{i\in I} p_i = \sum_{i\in J} p_i
\quad\Longrightarrow\quad
I=J.
\]
\end{enumerate}
Importantly, the  set $\mathcal A$ is dense in $\Pp_2(\R^n)$. Indeed, finitely supported distributions are dense in $W_2$, and the two additional properties can be enforced by arbitrarily small perturbations of any finitely supported distribution.  More precisely, the spanning condition can be ensured by splitting a support point into $n+1$ affinely independent points placed arbitrarily close together. The distinct subset-sum condition can be achieved by perturbing the weights: the configurations violating this property are characterized by a finite union of proper affine hyperplanes in the probability simplex, whose complement is therefore dense. 

\medskip
\noindent (a) Meagreness of \(\mathcal S_{\mathrm{cyc}}\). We prove that $\mathcal S_{\mathrm{cyc}}$ is meagre by writing it as a countable union of closed sets, each disjoint from $\mathcal A$. Since $\mathcal A$ is dense, it follows that each of these sets is nowhere dense. To define these sets, for \(k\ge 2\) and \(q\ge 1\), let
\[
K_{q}
:=
\Bigl\{
A\in \R^{2\times n}
:
\|A\|_{F}=1,
\|A_{1\cdot}\|\ge q^{-1},
\|A_{2\cdot}\|\ge q^{-1}
\Bigr\},
\]
and write \(R_k:=R_{2\pi/k}\). Also define
\[
\overline{\mu}
:=
(x\mapsto x-m(\mu))_\sharp \mu.
\]
and
\[
C_{k,q}
:=
\Bigl\{
\mu\in \Pp_2(\R^n)
:
\exists A\in K_q
\text{ such that }
A_\sharp \overline{\mu}
=
(R_k)_\sharp(A_\sharp \overline{\mu})
\Bigr\}.
\]
By definition,  
\[
\Phi^{-1}\left( \mathcal S_{\mathrm{cyc}}\cap \mathcal E^{\mathrm{EnKU}}_0 \right) =  \left(\bigcup_{k\ge 2}\bigcup_{q\ge 1} C_{k,q}\right)\times \Pp_{2,\mathrm{nd}}(\R^m)\times \R^{n\times m}. 
\]
Therefore, to conclude meagreness of $\mathcal S_{\mathrm{cyc}}\cap \mathcal E^{\mathrm{EnKU}}_0$, it suffices to show that each $C_{k,q}$ is closed and disjoint from $\mathcal A$.  To prove closedness, let $\mu_\ell \in C_{k,q}$ with $\mu_\ell \rightarrow \mu$ in $W_2$. By definition of $C_{k,q}$, for each $\ell$ there exists $A_\ell \in K_q$ such that
\[
(A_\ell)_\sharp \overline{\mu_\ell}
=
(R_k)_\sharp\big((A_\ell)_\sharp \overline{\mu_\ell}\big).
\]
Since $K_q$ is compact, there exists a subsequence (not relabeled) such that $A_\ell \rightarrow A \in K_q$.  Because \(K_q\) is compact and the map \((\mu,A)\mapsto A_\sharp \overline{\mu}\) is continuous in \(W_2\), we can take limits of both sides of the equation and obtain $(A)_\sharp \overline{\mu}
=
(R_k)_\sharp\big((A)_\sharp \overline{\mu}\big).$ This shows  \(C_{k,q}\) is closed. Next, we show that \(C_{k,q} \cap \mathcal A = \varnothing\). Indeed, fix \(\mu\in \mathcal A\), and suppose that \(\mu\in C_{k,q}\). By definition, for some \(A\in K_q\), the finitely supported distribution \(A_\sharp \overline{\mu}\) is invariant under the nontrivial rotation \(R_k\). Write $S$ for the support and 
\[
A_\sharp \overline{\mu}=\sum_{y\in S} m_y \delta_y.
\]
Each mass \(m_y\) is a subset sum of the weights \(p_i\). Rotational invariance gives
\[
m_y=m_{R_k y}\qquad \text{for all }y\in S.
\]
Since all subset sums of the \(p_i\) are distinct, equality of masses implies equality of the corresponding subsets of indices in the underlying finite particle distribution $\mu$. In particular, \(R_k y=y\) for every support point \(y\). But a nontrivial planar rotation fixes only the origin, so \(S=\{0\}\). Therefore
\[
A(x_i-m(\mu))=0
\qquad\text{for all }i,
\]
which contradicts the fact that the centered support spans \(\R^n\) and \(A\neq 0\). Thus \(\mathcal A\cap C_{k,q}=\varnothing\).

\medskip
\noindent (b) Meagreness of \(\mathcal S_{\mathrm{dec}}\). For \(N\ge 1\), define
\[
L_N
:=
\Bigl\{
(v,w,\lambda)\in \C^n\times \C^m\times \C
:
\|v\|=1,
\|w\|\le N,
|\lambda|\le 1-\tfrac1N
\Bigr\},
\]
and
\[
X_N:=\Bigl\{
(\mu,\nu)\in \Pp_2(\R^n)\times \Pp_{2,\mathrm{nd}}(\R^m)\big|
\int|x|^2d\mu(x)\le N,
\int|y|^2d\nu(y)\le N
\Bigr\}.
\]
We then set \(D_N\subset X_N\) to be the set of all pairs \((\mu,\nu)\in X_N\) such that for some \((v,w,\lambda)\in L_N\),
\[
v^\top \overline Z
\stackrel d=
\sum_{j=0}^\infty \lambda^j w^\top \overline Y^{(j)},
\]
where \(Z\sim \mu\), \(Y^{(j)}\stackrel{\text{i.i.d.}}{\sim}\nu\), and all variables are independent. Since 
\[
\Phi^{-1}(\mathcal S_{\mathrm{dec}}\cap \mathcal E^{\mathrm{EnKU}}_0) \subseteq 
\left(\bigcup_{N\ge 1} D_N\right)\times \R^{n\times m}
\]
it suffices to show that every $D_N$ is closed in the relative topology of $\Pp_2(\R^n)\times \Pp_{2,\mathrm{nd}}(\R^m)$ and has empty interior. Define
\[
\Theta : \C^n\times \Pp_2(\R^n)\rightarrow \Pp_2(\C),
\qquad
\Theta(v,\mu):=\Law(v^\top \overline Z),
\quad Z\sim\mu,
\]
and
\[
\Omega : L_N\times \{\nu\in \Pp_{2,\mathrm{nd}}(\R^m)\big| \int |y|^2 d\nu(y)\le N\}\rightarrow \Pp_2(\C)
\]
by
\[
\Omega(v,w,\lambda,\nu)
:=
\Law\!\left(\sum_{j=0}^{\infty} \lambda^j w^\top \overline{Y}^{(j)}\right),
\qquad
Y^{(j)}\stackrel{\text{i.i.d.}}{\sim}\nu.
\]
We show that \(D_N\) is closed. Let \((\mu_\ell,\nu_\ell)\in D_N\) with \((\mu_\ell,\nu_\ell)\rightarrow(\mu,\nu)\) in \(W_2\). For each \(\ell\), choose \((v_\ell,w_\ell,\lambda_\ell)\in L_N\) such that
\begin{equation}
\label{eq:DN-sequence}
\Theta(v_\ell,\mu_\ell)=\Omega(v_\ell,w_\ell,\lambda_\ell,\nu_\ell).
\end{equation}
Since \(L_N\) is compact, after passing to a subsequence we may assume 
$(v_\ell,w_\ell,\lambda_\ell)\rightarrow (v,w,\lambda)\in L_N.$  We claim closedness, i.e.
\[
\Theta(v,\mu)=\Omega(v,w,\lambda,\nu). 
\]
The left-hand side of \eqref{eq:DN-sequence} converges to $\Theta(v,\mu)$ because \(\Theta\) is continuous.  By uniqueness of the limit, it is enough to show that $\Omega(v_\ell,w_\ell,\lambda_\ell,\nu_\ell)\rightarrow \Omega(v,w,\lambda,\nu)$ in $W_2$. Let
\[
X_\ell:=\sum_{j=0}^\infty \lambda_\ell^j w_\ell^\top \overline Y_\ell^{(j)},
\qquad
X:=\sum_{j=0}^\infty \lambda^j w^\top \overline Y^{(j)},
\]
where \((Y_\ell^{(j)},Y^{(j)})_{j\ge 0}\) are i.i.d.\ copies of an optimal coupling of \(\nu_\ell\) and \(\nu\). Then
\[
W_2^2(\Omega(v_\ell,w_\ell,\lambda_\ell,\nu_\ell), \Omega(v,w,\lambda,\nu))
\le
\|X_\ell-X\|^2_{L^2}
\]
where we write $\|X\|_{L^2}^2:= \E\|X\|_2^2$. Decomposing each summand as
$$
\lambda_\ell^j w_\ell^\top \overline Y_\ell^{(j)}
-
\lambda^j w^\top \overline Y^{(j)}
=
\lambda_\ell^j w_\ell^\top(\overline Y_\ell^{(j)}-\overline Y^{(j)})
+
\lambda_\ell^j (w_\ell-w)^\top \overline Y^{(j)} 
+
(\lambda_\ell^j-\lambda^j) w^\top \overline Y^{(j)},
$$
the triangle inequality gives
\[
W_2(\Omega(v_\ell,w_\ell,\lambda_\ell,\nu_\ell), \Omega(v,w,\lambda,\nu))\le \text{I}_\ell + \text{II}_\ell + \text{III}_\ell
\]
where elementary bounds show 
\begin{align*}
\text{I}_\ell
&:=
\sum_{j=0}^\infty |\lambda_\ell|^j
\bigl\|w_\ell^\top(\overline Y_\ell^{(j)}-\overline Y^{(j)})\bigr\|_{L^2} \le
{2N^2}W_2(\nu_\ell,\nu),\\
\text{II}_\ell
&:=
\sum_{j=0}^\infty |\lambda_\ell|^j
\bigl\|(w_\ell-w)^\top \overline Y^{(j)}\bigr\|_{L^2} \le
{N}^{3/2}\|w_\ell-w\|_2,\\
\text{III}_\ell
&:=
\sum_{j=0}^\infty |\lambda_\ell^j-\lambda^j|
\bigl\|w^\top \overline Y^{(j)}\bigr\|_{L^2} \le {N}^{7/2}|\lambda_\ell-\lambda|.
\end{align*}
where $r:=1-\frac1N.$ 
In particular, 
\[
W_2(\Omega(v_\ell,w_\ell,\lambda_\ell,\nu_\ell), \Omega(v,w,\lambda,\nu))
\le
\|X_\ell-X\|_{L^2}\rightarrow 0.
\]
We now show that each \(D_N\) has empty interior. Let \(\mathcal V\subset \Pp_{2,\mathrm{nd}}(\R^m)\) denote the set of laws with smooth density. This set is dense in \(\Pp_{2,\mathrm{nd}}(\R^m)\), for instance by convolution with a nondegenerate Gaussian of arbitrarily small variance.  Therefore, $\mathcal A\times \mathcal V$ is dense in $\Pp_{2}(\R^n) \times \Pp_{2,\mathrm{nd}}(\R^m)$ and to show that $D_N$ has empty interior it  suffices to prove that $D_N \cap (\mathcal A\times \mathcal V ) =\varnothing$.  Fix \((\mu,\nu)\in \mathcal A\times \mathcal V\). If \(v\neq 0\) and letting  $Z\sim \mu$, \(v^\top \overline{Z}\) is a nondegenerate finite atomic law, because \(\mu\) is finitely supported and the centered support of \(\mu\) spans \(\R^n\). On the other hand, if \(w=0\), then for $Y \sim \nu$
\[
\sum_{j=0}^{\infty} \lambda^j w^\top \overline{Y}^{(j)} \equiv 0
\]
is degenerate. If \(w\neq 0\), then \(w^\top \overline{Y}\) is non-atomic since \(\nu\) has a smooth density, and therefore
\[
\sum_{j=0}^{\infty} \lambda^j w^\top \overline{Y}^{(j)}
=
w^\top \overline{Y}^{(0)}
+
\sum_{j=1}^{\infty} \lambda^j w^\top \overline{Y}^{(j)}
\]
is also non-atomic as the convolution of a non-atomic law with an independent law. Thus equality in distribution of these two is impossible and
\[
(\mathcal A\times \mathcal V)\cap D_N=\varnothing.
\]
\end{proof}

\begin{remark}
\label{rem:nldec_meagre}
\fjedit{We expect that the class \(\mathcal S_{\mathrm{nl-dec}}\), introduced in \Cref{section:beyond_enkf}, is meagre in an appropriate topology on a parametrized version of \(\mathcal F\). A natural choice is obtained by representing \(X=Z+d(Y)\), \(Z\perp Y\), with \(\E Z=0\), and parametrizing \(\mathcal F\) by \((\nu,\eta)=(\Law(Z),\Law(d(Y),Y))\), equipped with the relative product \(W_2\)-topology on \(\Pp_2^0(\R^n)\times\Gamma_2\), where \(\Gamma_2\) denotes the space of square-integrable graph laws $(d,I)_\sharp \mu$, $\mu \in \Pp_2(\R^m)$. We do not attempt to prove this conjecture here.}
\end{remark}